\documentclass[11pt]{amsart}
\usepackage{graphicx}
\usepackage{color}
\usepackage[small,bf]{caption}

\usepackage{eucal,amsmath,amssymb,amsthm,amsfonts}
\usepackage{epsfig}

\renewcommand{\ge}{\geqslant}

\renewcommand{\le}{\leqslant}

\newtheorem{theorem}{Theorem}
\newtheorem{lemma}[theorem]{Lemma}

\newtheorem{corollary}[theorem]{Corollary}

\newcommand{\N}{{\mathbb N}}
\newcommand{\R}{{\mathbb R}}

\renewcommand{\tilde}{\widetilde}

\numberwithin{equation}{section}

\newcommand{\CC}     {\mathcal{C}}

\newcommand{\J}     {\mathcal{J}}

\newcommand{\G}     {\mathcal{S}}
\newcommand{\PPara}     {\mathcal{P}}

\begin{document}

\title[Nonlocal minimal surfaces]{Regularity
properties\\
of nonlocal minimal surfaces\\
via limiting arguments}

\author{Luis Caffarelli}
\address{
University of Texas at Austin\\
Department of Mathematics \\
1 University Station C1200\\
Austin, TX 78712-0257 (USA)
}
\email{caffarel@math.utexas.edu}   
\author{Enrico Valdinoci}
\address{
Dipartimento di Matematica\\
Universit\`a di Roma Tor Vergata\\
Via della Ricerca Scientifica, 1\\
I-00133 Roma (Italy)
}
\email{enrico@math.utexas.edu}  

\begin{abstract}
We prove an improvement of flatness result for
nonlocal minimal surfaces which is independent
of the fractional parameter~$s$ when~$s\rightarrow 1^-$.

As a consequence,
we obtain that all the nonlocal minimal
cones are flat and that all the nonlocal minimal
surfaces are smooth when the dimension of
the ambient space is less or equal than~$7$
and~$s$ is close to~$1$.
\end{abstract}

\maketitle

The purpose of this paper is to study some
regularity properties of nonlocal minimal surfaces
as they approach the classical minimal surfaces.

Let~$n\ge2$ and~$s\in(0,1)$. Given two non-overlapping (measurable)
subsets $A$ and~$B$ of~$\R^n$, we define
$$ {\mathcal{L}}(A,B):= \int_A\int_B
\frac{1}{|x-y|^{n+s}}\,dy\,dx.$$
Given a bounded open set~$\Omega \subset\R^n$
and a set~$E\subseteq\R^n$,
we let
$$ \J_s (E,\Omega):=
{\mathcal{L}}({E\cap\Omega},{(\CC 
E)\cap \Omega})+{\mathcal{L}}({E\cap\Omega},{(\CC
E)\cap (\CC\Omega)})+{\mathcal{L}}({E\cap(\CC\Omega)},{(\CC
E)\cap \Omega}).$$
We say that~$E$ is~$s$-minimal in~$\Omega$
if
for any~$\tilde E\subseteq\R^n$ for which~$\tilde E
\cap (\CC \Omega)=E\cap (\CC \Omega)$ one has that 
$$ \J_s (E,\Omega)\le \J_s (\tilde E,\Omega).$$    
That is,~$E$ is~$s$-minimal if it
minimizes the functional among\footnote{Following a standard
convention in geometric measure theory, all the sets
will be implicitly assumed to contain their measure theoretic interior
and to lie outside their measure theoretic exterior -- this
is possible up to changing a set with a 
set of zero Lebesgue measure, which does not affect
the functional~$\J_s$. More explicitly, if we set
\begin{eqnarray*}&& E_{\mathcal{I}}
:=\{ x\in E {\mbox{ s.t. }} \exists r>0 {\mbox{ s.t. }} | (\CC E)\cap 
B_r(x)|=0\}\\ {\mbox{and }}&&
E_{\mathcal{E}}
:=\{ x\in E {\mbox{ s.t. }} \exists r>0 {\mbox{ s.t. }} | E\cap
B_r(x)|=0\},\end{eqnarray*}
we take the convention that~$ E_{\mathcal{I}}\subseteq E$ and~$E\cap 
E_{\mathcal{E}}=\varnothing$.}
competitors which agree outside~$\Omega$.  

The functional~$\J_s$ has been recently introduced in~\cite{CRS}
as a model for nonlocal minimal surfaces, and its
relation with the classical minimal surfaces
has been established in~\cite{CV,Amb}, both in the geometric sense
and in the Gamma--convergence framework.

Besides their neat geometric motivation,
such nonlocal minimal surfaces also arise
as limit interfaces of nonlocal phase segregation
problems, see~\cite{SaV1, SaV2}.

The main difficulty in the framework we consider is, of course,
the nonlocal aspect of the contributions
in the functional.
The counterpart of this difficulty, however, is given by the fact that
the functional is well defined for every (measurable) set
-- in particular,
there is no need to introduce Caccioppoli sets in this
case.
Nevertheless,
in spite of the results of~\cite{CRS, CV, Amb},
several regularity issues for $s$-minimizers
are still open.

The purpose of this paper is to develop
some regularity theory when~$s$ is close
to~$1$ by a compactness argument,
taking advantage of the 
regularity
theory of the classical minimal surfaces.
Our main result is the following improvement of flatness:

\begin{theorem}\label{MAIN}
Let~$s_o\in (0,1)$, $\alpha\in(0,1)$ and~$s\in[s_o,1)$.
Let~$E$ be~$s$-minimal in~$B_1$. There exists~$\varepsilon_\flat>0$,
possibly depending on~$n$, $s_o$ and~$\alpha$, but independent
of~$s$, such that if
\begin{equation}\label{main trap}
\partial E\cap B_1\subseteq \{ |x\cdot e_n|\le \varepsilon_\flat\}
\end{equation}
then~$\partial E$ is a~$C^{1,\alpha}$-graph 
in the~$e_n$-direction.
\end{theorem}

The crucial part of Theorem~\ref{MAIN} is that
its flatness threshold~$\varepsilon_\flat$ is independent of~$s$
as~$s\rightarrow1^-$: in fact, for a fixed~$s$,
an improvement of flatness whose threshold
depends on~$s$ has been obtained in~\cite{CRS}
(see Theorem~6.1 there).
The techniques used to prove Theorem~\ref{MAIN}
(hence to obtain a threshold independently of~$s$
as~$s\rightarrow1^-$)
are a uniform
measure estimate for the oscillation,
and a Calder{\'o}n--Zygmund iteration.
Both these tools have somewhat a classical flavor, but
they need to be appropriately, and deeply, modified here:
in particular, some fine estimates
performed in~\cite{CS} turn out to be very useful
here in order to obtain bounds that are independent of~$s$,
and the iteration is not straightforward, but it has to
distinguish two cases according to the size of the cubes
involved, and the technical difficulties arising in the course
of the proof turn out to be quite challenging.

As a consequence of Theorem~\ref{MAIN},
we obtain several regularity and rigidity results
for $s$-minimal surfaces, such as:

\begin{theorem}\label{-2-}
Let~$n\le 7$.

There exists~$\epsilon_o>0$ such that if~$s\in(1-\epsilon_o,1)$
then any $s$-minimal cone is a hyperplane.
\end{theorem}

\begin{theorem}\label{-3-}
Let~$n\le 7$.

There exists~$\epsilon_o>0$ such that if~$s\in(1-\epsilon_o,1)$
then any $s$-minimal set is locally a $C^{1,\alpha}$-hypersurface.
\end{theorem}

\begin{theorem}\label{GG1}
Let~$n=8$.

There exists~$\epsilon_o>0$ such that if~$s\in(1-\epsilon_o,1)$
then any $s$-minimal set is locally~$C^{1,\alpha}$,
everywhere except, at most, at countably many isolated points.
\end{theorem}

\begin{theorem}\label{GG2}
There exists~$\epsilon_o>0$ such that if~$s\in(1-\epsilon_o,1)$
then any $s$-minimal set is locally~$C^{1,\alpha}$
outside a closed set~$\Sigma$, with~$
{\mathcal{H}}^d(\Sigma)=0$ for any~$d>n-8$.
\end{theorem}

For other recent regularity results for
nonlocal minimal surfaces see~\cite{Bar, SaV3}.
The organization of the paper is
displayed by the following table:

\tableofcontents

\section{Notation}\label{notation}

A point~$x\in \R^n$ will be often written
in coordinates as~$x=(x',x_n)\in\R^{n-1}
\times\R$.

The complement of a set~$\Omega\subseteq\R^n$
will be denoted by~$\CC \Omega:=\R^n\setminus\Omega$.
For any $P\in\R^n$ and $\rho>0$,
we define the cylinder
$$ K_\rho(P):= \{ |x'-P'|<\rho\}\times
\{ |x_n - P_n|<\rho\}.$$
We also set~$K_\rho:=K_\rho(0)$.

The $(n-1)$-dimensional
cube of side~$R$ centered at~$x_o'\in\R^{n-1}$ will
be denoted by~$Q_R(x_o)$.

If~$\nu\in{\rm S}^{n-1}$, given~$x\in\R^n$, we define its
projection along~$\nu$, that is~$\pi_\nu x:=
x-(x\cdot \nu)\nu$.


Given a set~$E\subset\R^n$, we denote
by~${d}_E (x)$ the signed
distance of a point~$x\in\R^n$; we will
take the sign convention that~${d}_E (x)\ge0$
if~$x\in\CC E$.

If~$\Sigma\subset\R^n$ is a $C^2$-portion of
hypersurface, we define~${\mathcal{H}}(P)$
to be the mean curvature of~$\Sigma$ at~$P$
(with the convention that~${\mathcal{H}}$
equals the sum
of all the principal curvatures).

The $k$-dimensional Lebesgue measure of a (measurable) set~$A\subseteq
\R^k$ will be denoted by~$|A|$.

We let~$\varpi$ be the $(n-2)$-dimensional Hausdorff measure of
the boundary of the~$(n-1)$-dimensional unit ball.

Often, we will denote by~$c$, $C$ a suitable
positive constant, that we 
allow ourselves the latitude of renaming at each
step of the computation.

\section{Proof of Theorem~\ref{MAIN}}

Now we start the proof of Theorem~\ref{MAIN},
which is based on several steps. 

First, we need
to approximate our $s$-minimal surface with a graph.
As soon as~$s$ approaches~$1$, a flat $s$-minimal surface
approach a classical, smooth, minimal surface,
and this will allow us to keep the Lipschitz norm
of this approximating graph under control.

Then, we perform an estimate on the detachment
of this graph from its tangent hyperplane: this bound
(together with a suitable auxiliary function and an estimate relating
the integral equation with the classical mean curvature
equation in the limit) provides
an Alexandrov-Bakelman-Pucci type theory
that controls the oscillation of the graph in measure.

This may be repeated at finer and finer scales via
dyadic decomposition, by possibly taking advantage
of the closeness to the smooth minimal surface
when the size of the cubes become too small. In this
way, one obtains a pointwise control on the oscillation
of the approximating graph (and so of the original
$s$-minimal surface), leading to the proof of Theorem~\ref{MAIN}.

Below are the full details or the proof.

\subsection{Building a graph via the distance function}

One of the difficulties of our framework
is that the~$s$-minimal surfaces we are dealing with
are not necessarily graphs. To get around this
problem, we follow an idea of~\cite{CC} and we consider 
level sets of the distance function in an appropriate scaling
(this may be seen as a sup-convolution technique).

For this, 
we recall the following classical
geometric observation
on the regularity of the level sets of
the distance function:

\begin{lemma}\label{CC1L}
Let~$E\subset\R^n$.
Assume that
\begin{equation}\label{Gt}
\{ x_n\le -\gamma\} \cap K_r\subseteq
E\cap K_r \subseteq \{ x_n\le \gamma\}
\cap K_r,\end{equation}
for some~$r>\gamma>0$.

Let~$\delta\in(0,r/4)$ and~$
{\mathcal{S}}^\pm:=\{ x\in \R^n {\mbox{ s.t. }} d_E(x)=
\pm\delta\}$.

Then, there exist~$c\in(0,1)$ and~$C\in(1,+\infty)$
such that if~$\gamma/\delta<c$
then~${\mathcal{S}}^\pm\cap K_{r-2\delta}$
is a Lipschitz graph in the~$n$th direction with
Lipschitz constant bounded by~$C\sqrt{\gamma/\delta}$.

Furthermore,~${\mathcal{S}}^-$
(resp.,~${\mathcal{S}}^+$) may be touched
at any point of~$K_{r-2\delta}$
by a tangent paraboloid from above (resp., below).
\end{lemma}

\begin{proof} We focus on~${\mathcal{S}}^-$, the
case of~${\mathcal{S}}^+$ being analogous.
We would like to show that
for any~$x$, $z\in {\mathcal{S}}^-\cap K_{r-2\delta}$
\begin{equation}\label{Ch7}
x_n-z_n\le C\sqrt{\frac{\gamma}{\delta}}\, |x'-z'|,
\end{equation}
from which the desired result follows
by possibly exchanging the roles of~$x$ and~$z$.

For this, we argue like this.
For any~$x\in{\mathcal{S}}^-\cap K_{r-2\delta}$,
the ball of radius~$\delta$
centered at~$x$
is tangent to~$\partial E$ at some point~$y(x)\in \partial E
\cap K_r$, and,
conversely, 
\begin{equation}\label{y503}{\mbox{the ball
of radius~$\delta$
centered at~$y(x)$
is tangent to~${\mathcal{S}}^-$ at~$x$.}}
\end{equation}
Let~$e_n:=(0,\dots,1)$. Since~$x+\delta e_n\in B_{
\delta} (x)$, we have that~$x+\delta e_n$
must lie 
in the closure of~$E$.
Hence,
by~\eqref{Gt},
\begin{equation}\label{1.3a}
x_n+\delta\le \gamma.\end{equation}
Similarly, since~$y(x)\in\partial E$, we obtain from~\eqref{Gt}
that
\begin{equation}\label{1.3a-b}
y_n(x)\ge -\gamma.\end{equation}
By~\eqref{1.3a} and~\eqref{1.3a-b},
\begin{equation}\label{X} {y_n(x)-x_n}
\ge \delta-{2\gamma}.\end{equation}
In the same way, we see that
\begin{equation}\label{Xz} {y_n(z)-z_n}
\ge \delta-{2\gamma}.\end{equation}
Now, if~$|x'-z'|\ge \sqrt{\gamma\delta}$, we use~\eqref{Gt}
and~\eqref{X} 
to deduce that
\begin{eqnarray*}
&& x_n-z_n\le (x_n-y_n(x))+|y_n(x)|+|y_n(z)|+|y_n(z)-z_n|
\\&&\qquad\le
(2\gamma-\delta)+\gamma+\gamma+|y(z)-z|
\\&&\qquad\le
(2\gamma-\delta)+\gamma+\gamma+\delta\\
&& \qquad =4\gamma\le 4\sqrt{\frac{\gamma}{\delta}}\,|x'-z'|,
\end{eqnarray*}
which proves~\eqref{Ch7} in this case.

So, we may focus on the case in which
\begin{equation}\label{90c}
|x'-z'|\le\sqrt{\gamma\delta}.\end{equation}
Then, from~\eqref{Xz},
$$ \delta^2=|y(z)-z|^2=|y'(z)-z'|^2+|y_n(z)-z_n|^2
\ge |y'(z)-z'|^2+(\delta-2\gamma)^2,$$
which gives
\begin{equation}\label{90i}
|y'(z)-z'|\le 2\sqrt{\gamma\delta}.\end{equation}
Hence
\begin{equation*}
|x'-y'(z)|\le |x'-z'|+|z'-y'(z)|\le 3\sqrt{\gamma\delta},
\end{equation*}
due to~\eqref{90c}
and~\eqref{90i}, and so, in particular,
\begin{equation}\label{90ttt}
|x'-y'(z)|\le \frac{\delta}{100}.
\end{equation}
So, we can define
\begin{equation}\label{Xz2}
p:= \big( x', y_n(z)-\sqrt{\delta^2-|y'(z)-x'|^2}\big).
\end{equation}
We observe that
\begin{equation}\label{190ttt}
p\in \partial B_\delta(y(z)).\end{equation}
Also, from~\eqref{X}
and~\eqref{Gt},
$$ y_n(z)-x_n\ge y_n(z)-y_n(x)+\delta-2\gamma\ge\delta-4\gamma>0.$$
Therefore,
by~\eqref{90ttt}, we have that~$x$ must be below~$B_\delta(y(z))$,
hence~\eqref{190ttt}
implies that
\begin{equation}\label{78yy0}
x_n\le p_n.\end{equation}
Now, we define~$P:=(p-y(z))/\delta$ and~$Z:=(z-y(z))/\delta$.
We observe that~$P$, $Z\in \partial B_1$, due to~\eqref{190ttt}.
Also, $P_n$, $Z_n\le0$, due to~\eqref{Xz}
and~\eqref{Xz2}. Moreover, $|P'|+|Z'|\le 1/50$
thanks to~\eqref{90i}, \eqref{90ttt}
and~\eqref{Xz2}.
As a consequence
$$ |P_n-Z_n|\le 100\, |P'-Z'|^2.$$
By scaling back, this gives that
$$ |p_n-z_n|\le\frac{100}{\delta}\,|p'-z'|^2=
\frac{100}{\delta}\,|x'-z'|^2 \le
100\sqrt{\frac{\gamma}{\delta}} \,|x'-z'|,$$
where~\eqref{90c} was used once again.
{F}rom this and~\eqref{78yy0}, we infer that
$$ x_n-z_n\le p_n-z_n\le 
100\sqrt{\frac{\gamma}{\delta}} \,|x'-z'|,$$
which gives~\eqref{Ch7} in this case too.

Then, the desired Lipschitz property is
a consequence of~\eqref{Ch7},
and the existence of a tangent paraboloid follows
from~\eqref{y503} (and, by~\eqref{X}, the touching occurs from above in this case).
\end{proof}

We point out that the Lipschitz bound $C\sqrt{\gamma/\delta}$
in Lemma \ref{CC1L} is optimal, as the example in Figure~1 shows.

\begin{figure}[htbp]
\begin{center}
\resizebox{13.2cm}{!}{\input{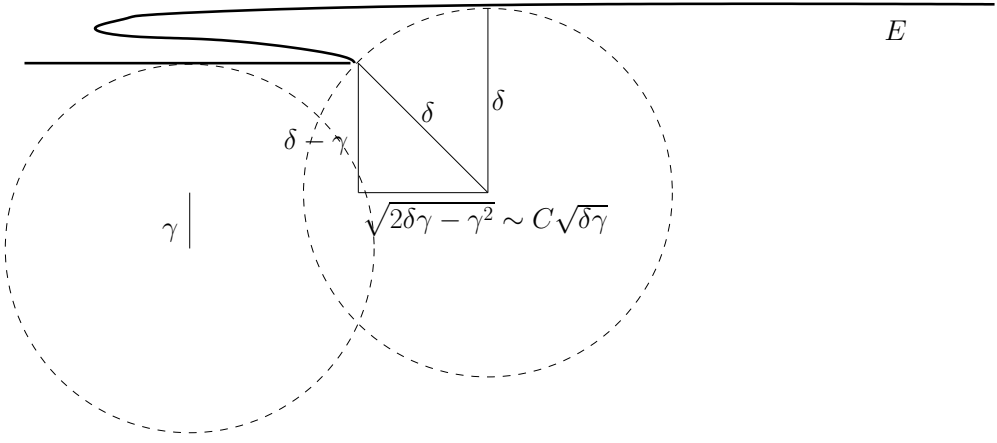}}
{\caption{\it Optimality of the Lipschitz constant
$\sqrt{\gamma/\delta}=\gamma/\sqrt{\delta\gamma}$ in Lemma \ref{CC1L}.}}
\end{center}
\end{figure}

A global version of Lemma~\ref{CC1L} is given by
the following result:

\begin{corollary}\label{CorCC}
Let~$E_\star\subseteq\R^n$. Suppose that~$\partial E_\star\cap K_2$
is a $C^{1,\alpha}$-graph in the $n$th direction, for some~$\alpha>0$,
and let~$M_\star$ be its $C^{1,\alpha}$-norm.

Then, there exists~$c_\star\in(0,1)$,
possibly depending on~$M_\star$,
such that the
following holds.

Let~$\gamma$, $\delta\in (0,1/4)$,~$E\subseteq\R^n$
and suppose that
\begin{equation}\label{n992}
{\mbox{$E\cap K_2$ lies in
a~$\gamma$-neighborhood of~$E_\star$.
}}\end{equation}
Let~$
{\mathcal{S}}^\pm:=\{ x\in \R^n {\mbox{ s.t. }} d_E(x)=
\pm\delta\}$.

Then,~${\mathcal{S}}^\pm\cap K_1$
is a Lipschitz graph in the~$n$th direction,
provided that~$\gamma/\delta<c_\star$, $\delta<c_\star
\gamma^{1/(1+\alpha)}$
and~$\gamma<c_\star$.

More precisely, there exists a constant~$C>1$ for which~$
{\mathcal{S}}^\pm\cap K_1$
is a Lipschitz graph in the~$n$th direction and
the Lipschitz norm of~${\mathcal{S}}^\pm\cap K_1$ is
controlled by~$C\sqrt{\gamma/\delta}+M_o$, where~$M_o$
is the Lipschitz norm
of~$\partial E_\star\cap K_2$.

Furthermore,~${\mathcal{S}}^-$
(resp.,~${\mathcal{S}}^+$) may be touched
at any point of~$K_{1-2\delta}$
by a tangent paraboloid from above (resp., below).
Finally, for any~$|x'|\le 1/2$,
\begin{equation}\label{s8822211a}
u^+(x')-u^-(x')\le 2(2+M_o)(\gamma+\delta)
.\end{equation}
\end{corollary}

\begin{proof} Since~$\partial E_\star\cap K_2$
is~$C^{1,\alpha}$, it separates with power~$(1+\alpha)$
from its tangent hyperplane, with
multiplicative constant~$M_\star$. 
Then, we take~$r:=
(\gamma/M_\star)^{1/(1+\alpha)}$ and we
cover~$\partial E_\star\cap K_2$
with cylinders~$K_r$, centered at points of~$\partial E_\star$
and rotated parallel to the
tangent plane of~$\partial E_\star$.

By construction,
in each of these cylinders, $\partial E_\star$ separates
no more than~$M_\star r^{1+\alpha}=\gamma$ from its
tangent hyperplane, 
and so~$E$ is~$2\gamma$-close
to such hyperplane. Therefore, Lemma~\ref{CC1L}
applies (with~$\gamma$ there replaced by~$2\gamma$).
Consequently, in each of these cylinders,~${\mathcal{S}}^\pm$
is a Lipschitz graph with respect to the normal
direction~$\nu$ of~$\partial E_\star$ (and its Lipschitz norm
is bounded by~$C\sqrt{\gamma/\delta}$ with respect to~$\nu$).

This proves the first part of Corollary~\ref{CorCC}.
It remains to prove~\eqref{s8822211a}.
For this, we fix~$|\bar x'|\le 1/2$
and we set~$P^\pm:=(\bar x',u^\pm(\bar x'))\in{\mathcal{S}}^\pm$.
Then, we take~$Q^\pm \in \partial E$ that realizes the distance,
i.e. $|P^\pm-Q^\pm|=\delta$.
By~\eqref{n992}, we find points~$R^\pm\in \partial E_\star$ such 
that~$|R^\pm-Q^\pm|\le\gamma$. 
Notice that
$$ |(R^\pm)' - (P^\pm)'|\le |(R^\pm)' - (Q^\pm)'|+
|(Q^\pm)' - (P^\pm)|\le \gamma+\delta.$$
Therefore, since~$(P^+)'=(P^-)'=u(\bar x)$, 
$$ |(R^+)' - (R^-)'|\le |(R^+)'-(P^+)'|+|(P^-)'-(R^-)'|\le
2 (\gamma+\delta).$$
So, since~$\partial E_\star$
is a Lipschitz graph,
$$ |R^+_n - R^-_n|\le M_o|(R^+)' - (R^-)'|
\le2 M_o(\gamma+\delta).$$
In particular,
$$ |R^+-R^-|\le 2(1+M_o)(\gamma+\delta)$$
and so
\begin{eqnarray*}
&& |P^+-P^-|
\\ &\le& |P^+-Q^+|+|Q^+-R^+|+|R^+-R^-|
+|R^--Q^-|+|Q^--P^-|
\\ &\le&
2(1+M_o)(\gamma+\delta)+2\gamma+2\delta,
\end{eqnarray*}
which gives~\eqref{s8822211a}.
\end{proof}

\subsection{Detachment from the tangent hyperplane}\label{S2p}


Next result is one of the cornerstones of our
procedure since it manages
to reconstruct a geometry similar to the one
obtained in Lemma~8.1 of~\cite{CS}. In spite of
its technical flavor, it basically states
under which conditions we can say that
a functions separates from
a tangent hyperplane
quadratically in a ring, independently of~$s$
as~$s\rightarrow 1^-$.

\begin{lemma}\label{S LEM} Fix~$\overline C\ge1$.
Let~$\varepsilon$, $R>0$ and~$\bar x'\in \R^{n-1}$.

Let~$u:\R^{n-1}\rightarrow\R$ be a Lipschitz function, with
\begin{equation}\label{gradient}
|\nabla u (x')|\le \overline C
\end{equation}
a.e.~$|x'-\bar x'|\le R$
and let~$\bar x_n:=u(\bar x')$, $\bar x:=(\bar x',\bar x_n)$
and~$E:=\{ x_n<u(x')\}$.

Assume that
\begin{equation}\label{the eq}
(1-s)\int_{B_{R}(\bar x)}\frac{\chi_E(y)-\chi_{\CC E}(y)}{|
\bar x-y|^{n+s}}\,dy
\le \frac{\varepsilon}{R^s}.
\end{equation} 
Suppose that there exists~$\PPara\in C^{1,1}(\R^{n-1})$
such that 
\begin{equation}\label{C1.1}
|\nabla \PPara(x')|+R\,|D^2\PPara(x')|\le\varepsilon
\end{equation}
a.e.~$|x'-\bar x'|\le R$,
\begin{equation}\label{above gamma}
{\mbox{$\PPara(\bar x')=u(\bar x')$
and $\PPara(x')\le u(x')$
in $|x'-\bar x'|\le R$.}}
\end{equation}
Then, there exists a constant~$C\ge1$, only depending
on~$n$ and~$\overline C$,
such that\footnote{The reader may compare~\eqref{star3}
here and~(8.1) in~\cite{CS}.
Notice that such an estimate, roughly speaking,
says that $u$ separates quadratically from
its tangent hyperplane in a ring, up
to a set with small density -- and the the constants
are independent of~$s$.

{F}rom this, a general geometric argument implies
a uniform quadratic detachment
in a whole ball with smaller radius
(see~(8.2) and~(8.3) in~\cite{CS}) and consequently
a linear bound on the image of the subdifferential
of the convex envelope (see~(8.4) in~\cite{CS}),
and this is the necessary ingredient
for the Alexandrov-Bakelman-Pucci theory to work
(see Sections~8, 9 and~10 in~\cite{CS}).
In our framework,~$u$ will be the level set of the distance
from an~$s$-minimal surface: we will add to
it the auxiliary function
of Section~\ref{S COM}
and consider the touching point of the convex envelope.
These points, by construction are touched from below
by a hyperplane, so~$u$ is touched from below
by a smooth function, which motivates the setting
of Lemma~\ref{S LEM}.}
the following result
holds, as long as~$\varepsilon\in(0,1/C)$.
There exists a $(n-1)$-dimensional
ring~$S_r:=\{ |x'-\bar x'|\in (r/C, r)\}$, with~$r\in(0,R]$,
such that, for any~$M>0$ 
we have
\begin{equation}\label{star3}
\frac{ \left| S_r \cap \Big\{
u(x')-\bar x_n-\nabla\PPara(\bar x')\cdot (x'-\bar x')>
\displaystyle\frac{
M\varepsilon r^2}{R}
\Big\}\right|}{\big| S_r\big|} \,\le\,
\frac{C}{M} .\end{equation}
\end{lemma}

\begin{proof} We consider the normal vector of
the graph of~$\PPara$
at~$\bar x'$, to wit
$$ \nu:=\frac{(-\nabla\PPara(\bar x'),1)}{
\sqrt{|\nabla\PPara(\bar x')|^2+1}}.$$
Let also
\begin{eqnarray*}
P&:=&\{ x_n<\PPara(x')\},\\
L&:=&\{x_n<\nabla\PPara(\bar x')\cdot (x'-\bar x')+\bar x_n\}
\\ {\mbox{and }}\;
A&:=&\bar x+\left\{ |x\cdot \nu|\le \frac{4\varepsilon}{R}
|\pi_\nu x|^2\right\}.\end{eqnarray*}
We recall that~$\pi_\nu$
is the projection along~$\nu$ (see Section~\ref{notation})
and we
notice that~$A$ is just the translation and the
rotation of the set
$$ \left\{ |x_n|\le \frac{4\varepsilon}{R}
|x'|^2 \right\}$$
and so, for any~$\rho>r>0$,
\begin{equation}\label{ST7}
\int_{B_\rho(\bar x)\setminus B_r(\bar x)}
{\chi_A(y)}\,dy
\le
\int_{|y'|\le \rho}\left[ \int_{|y_n|\le ({4\varepsilon}/{R})
|y'|^2}\,dy_n \right]\,dy'
\le\frac{C\varepsilon\rho^{n+1}}{ R}.
\end{equation}
On the other hand, since $L$ is a halfspace passing
through~$\bar x$,
the following cancellations hold:
\begin{equation}\label{A one}
\int_{B_\rho(\bar x)\setminus B_r(\bar x)}
{\chi_L(y)-\chi_{\CC L}(y)}\,dy =0
\;{\mbox{ and }}\;
\int_{B_\rho(\bar x)\setminus B_r(\bar x)}
\frac{\chi_L(y)-\chi_{\CC L}(y)}{|\bar x-y|^{n+s}}\,dy =0.
\end{equation}
Moreover, by~\eqref{above gamma}, we have that~$P\subseteq E$, thus
\begin{equation}\label{pre2.9a}
\chi_E\ge\chi_P
\;{\mbox{ and so }}\;
\chi_{\CC E}\le
\chi_{\CC P}.
\end{equation}
Also, the quadratic detachment of~$\PPara$ from
its tangent plane given by~\eqref{C1.1} implies that~$
(L\setminus A)\cap B_R\subseteq P\cap B_R$ and $(\CC P)
\cap B_R\subseteq((\CC L)\cup A)\cap B_R$.
Therefore, in~$B_R$,
\begin{equation}\label{pre2.9}
{\mbox{$\chi_L-\chi_A\le \chi_{L\setminus A}
\le\chi_P$ and~$\chi_{\CC P}\le \chi_{(\CC L)\cup A}
\le \chi_{\CC L}+\chi_A$.}}\end{equation}
So, from~\eqref{pre2.9a}
and~\eqref{pre2.9}, we obtain that,
in~$B_R$,
\begin{equation}\label{CHI}\chi_E-\chi_{\CC E}\ge 
\chi_P-\chi_{\CC P}\ge \chi_L-\chi_{\CC L} -2\chi_A.
\end{equation}
Now, for any~$m\in\N$, let
\begin{eqnarray*}
r_m&:=&\frac{R}{\big( (2+\overline C) n\big)^{m}},\\
R_m&:=& B_{r_m}(\bar x)\setminus B_{r_{m+1}}(\bar x)\\
{\mbox{and }}
b_m&:=&\int_{R_m}\frac{\chi_E(y)-\chi_{\CC E}(y)}{|
\bar x-y|^{n+s}}\,dy.\end{eqnarray*}
Here above~$\overline C$ is the one fixed in the statement of
Lemma~\ref{S LEM}.
We claim that there exists~$m\in\N$ such that
\begin{equation}\label{the 1st}
b_m\le \frac{C_o\varepsilon r_m^{1-s}}{R} ,
\end{equation}
for a suitable constant~$C_o\ge1$.
The proof is by contradiction: if not, we have
\begin{equation*}\begin{split}
& \int_{B_R(\bar x)}\frac{\chi_E(y)-\chi_{\CC E}(y)}{|
\bar x-y|^{n+s}}\,dy\\
& \qquad=\sum_{m=0}^{+\infty} b_m
\ge 
\frac{C_o\varepsilon}{R} \sum_{m=0}^{+\infty} r_m^{1-s}=
\frac{C_o\varepsilon}{R^{s}} \sum_{m=0}^{+\infty} \big(
(2+\overline C)n 
\big)^{-(1-s)m}
\\ &\qquad =
\frac{C_o\varepsilon}{R^{s}}\cdot\frac{1}{1-\big(
(2+\overline C) n
\big)^{-(1-s)}}
>\frac{C_o\varepsilon}{R^{s}}\cdot\frac{1}{C(1-s)}
\end{split}\end{equation*}
for some~$C>0$.
This is in contradiction with~\eqref{the eq} if~$C_o$ is large,
and so~\eqref{the 1st} is established. {F}rom now on,~$m$
will be the one given by~\eqref{the 1st}, and~$C_o$ will be simply~$C$
(and, as usual, we will take the freedom of renaming $C$ line after line).

Now, we make use of~\eqref{CHI}, \eqref{A one} and~\eqref{ST7}
to obtain that
\begin{equation*}\begin{split}
& \int_{R_m} \Big( \chi_E(y)-\chi_{\CC E}(y)\Big)
\left( \frac{1}{|\bar x-y|^{n+s}}
-\frac{1}{r_m^{n+s}} \right)
\,dy\\ &\qquad\ge 
\int_{R_m} \Big( \chi_L(y)-\chi_{\CC L}(y)-2\chi_A(y)\Big)
\left( \frac{1}{|\bar x-y|^{n+s}}
-\frac{1}{r_m^{n+s}} \right)
\,dy\\ &\qquad=
\int_{R_m} \Big( -2\chi_A(y)\Big)
\left( \frac{1}{|\bar x-y|^{n+s}}
-\frac{1}{r_m^{n+s}} \right)\,dy\\ &\qquad\ge
-2
\int_{R_m} 
\frac{\chi_A(y)}{|\bar x-y|^{n+s}}\,dy
\\ &\qquad \ge-\frac{C}{r_m^{n+s}}
\int_{R_m} \chi_A(y)\,dy
\\ &\qquad \ge-\frac{C\varepsilon r_m^{1-s}}{ R}.
\end{split}\end{equation*}
Combining this
with~\eqref{the 1st}, we conclude that
\begin{equation*}\begin{split}
\frac{|E\cap R_m|-|(\CC E)\cap R_m|}{r_m^{n+s}}
\,=& \,\int_{R_m}\frac{\chi_E(y)-\chi_{\CC E}(y)}{r_m^{n+s}}\,dy
\\ =& \,b_m-
\int_{R_m} \Big( \chi_E(y)-\chi_{\CC E}(y)\Big)
\left(\frac{1}{|\bar x-y|^{n+s}}
-\frac{1}{r_m^{n+s}} \right)
\\ \le& \frac{C\varepsilon r_m^{1-s}}{R}
\end{split}\end{equation*}
that is
\begin{equation}\label{star1} |E\cap R_m|-|(\CC E)\cap R_m|\le
\frac{C\varepsilon
r_m^{n+1}}{R}.\end{equation}

Now we prove
that
\begin{equation}\label{star2} 
\int_{ \{ r_{m+1}\le |x'-\bar x'|\le r_m/(\overline C\sqrt n) \} } 
u(x')-\bar x_n-\nabla\PPara(\bar x')\cdot (x'-\bar x')\, dx'
\le
\frac{C\varepsilon r_m^{n+1}}{R}.
\end{equation}
To this scope, we observe that
$$ K_{r_m/\sqrt{n} }\subseteq B_{r_m}\subseteq K_{r_m}$$
and~$r_{m+1}< r_m/(\overline C\sqrt{n})$. Hence
\begin{equation}\label{9-11}
S_m:=\big\{ r_{m+1}< |x'-\bar x'|< r_m/\sqrt{n}\big\} \times
\big\{ |x_n-\bar x_n|< r_m/\sqrt{n}\big\} \,\subseteq \,R_m.
\end{equation}
Of course, no confusion should arise between~$S_m$ here and~$S_r$
in the statement of Lemma~\ref{S LEM}.

Let~$\alpha:=\chi_E-\chi_{L}=\chi_{\CC L}-
\chi_{\CC E}$. We recall that
\begin{equation}\label{9-11-bis}
{\mbox{$\alpha+\chi_A\ge 0$ in $R_m$,}}
\end{equation}
due to~\eqref{pre2.9a} and~\eqref{pre2.9}. 

Accordingly,
by~\eqref{ST7}, \eqref{A one},
\eqref{9-11} and~\eqref{9-11-bis},
\begin{equation}\label{9-12}\begin{split}
&|E\cap R_m|-|(\CC E)\cap R_m|
\\ =\,& \int_{R_m} {\chi_E(y)-\chi_{\CC E}(y)} \,dy -0
\\ =\,& \int_{R_m} {\chi_E(y)-\chi_{\CC E}(y)} \,dy
-\int_{R_m} {\chi_{L}(y)-\chi_{\CC L}(y)} \,dy
\\ =\,&
2 \int_{R_m} {\alpha(y)}\,dy\\ =\,&
2 \int_{R_m} {\alpha(y)}+\chi_A(y) \,dy-2\int_{R_m}\chi_A(y)\,dy
\\ \ge\,& 2 \int_{S_m} {\alpha(y)}+\chi_A(y)
\,dy-\frac{C\varepsilon r_m^{n+1}}{R}.
\end{split}\end{equation}
Now, 
we use~\eqref{gradient} and~\eqref{C1.1}
to see that, if~$|y'-\bar x'|<r_m/
(\overline C\sqrt n)$, we have
\begin{equation}\label{9-13}\begin{split}
& |\nabla\PPara(\bar x')\cdot(y'-\bar x')|\le
|y'-\bar x'|<r_m/\sqrt{n} \\ {\mbox{and }}\;&
|u(y')-\bar x_n| =|u(y')-u(\bar x')|\le \overline C
|y'-\bar x'|<r_m/\sqrt{n}.
\end{split}\end{equation}
Hence, fixed~$y'$, with~$|y'-\bar x'|\in \big(r_{m+1}, r_m/
(\overline C\sqrt{n})\big)$
we see that~$\alpha(y',y_n)=1$ 
when~$(y',y_n)$ is trapped between~$E$ and~$\CC L$
(notice that it cannot exit~$S_m$ from either the
top or the bottom, by~\eqref{9-13}), i.e.,
when
$$\bar x_n+\nabla\PPara(\bar x'
)\cdot (\bar x')(y'-\bar x')\le y_n<u(y').$$
So, recalling~\eqref{9-11-bis}
and integrating first in~$dy_n$, we have that
$$ \int_{S_m} {\alpha(y)}+\chi_A(y) \,dy \,\ge\,
\int_{ \big\{|y'-\bar x'|\in (r_{m+1}, r_m/(\overline C\sqrt{n}))
\big\} }\Big(
u(y') - \bar x_n-\nabla\PPara(
\bar x')\cdot(x'-\bar x')\Big)^+\,dy'.$$
This,~\eqref{9-12} and~\eqref{star1} imply~\eqref{star2}.

Then,~\eqref{star3} follows from~\eqref{star2}
and the Chebyshev Inequality, taking~$r:=r_m/(\overline C\sqrt{n})$,
$S_r:=
\{ |x'-\bar x'|\in (r_{m+1}, r_m/(\overline C\sqrt{n}))\}$
and noticing that~$|S_r|
\sim r_m^{n-1}$ (remember that~$S_r\subset\R^{n-1}$).
\end{proof}

\subsection{The mean curvature as a limit equation}

In this section, we show that the integral equation
of~$s$-minimal surfaces converges, in a somewhat
uniform way, to the classical mean curvature equation
as~$s\rightarrow1^-$,
and we remark that the estimates
improve as the surfaces gets flatter
and flatter (see~\cite{Aba} for a more detailed discussion
on nonlocal curvatures).
An estimate of this kind will be
useful in the computation of the
forthcoming Lemma~\ref{barrier}.

\begin{lemma}\label{limit curvature}
Let~$s\in[1/10,1)$.
Let~$\alpha\in(0,1)$.
Let~$F\subset\R^n$,
$x_o\in\partial F$,
and suppose that~$\partial F\cap 
B_1(x_o)$
is a $C^{2,\alpha}$-graph in some direction,
with~$C^{2,\alpha}$-norm bounded by some~$M>0$.

Then, there exists~$C\ge 1$, only depending on~$\alpha$
and~$n$, such that
\begin{equation}\label{LAC1}
\left|{\mathcal{H}}(x_o)-\frac{(n-1)(1-s)}{\varpi}
\int_{B_r}\frac{\chi_F(y)-
\chi_{\CC F}(y)}{|x_o-y|^{n+s}}\,dy\right|\le 
\frac{CM(1-s)}{r},
\end{equation}
where~${\mathcal{H}}$ is the mean curvature (see Section~\ref{notation}) and
\begin{equation}\label{r definition}
r:=\min\left\{\frac{1}n,\frac{1}{2M}\right\}.
\end{equation}
In particular, if~$M\in (0,1]$, 
\begin{equation}\label{LAC3}
\left|{\mathcal{H}}(x_o)-\frac{(n-1)(1-s)}{\varpi}
\int_{B_r}\frac{\chi_F(y)-
\chi_{\CC F}(y)}{|x_o-y|^{n+s}}\,dy\right|\le
{C M(1-s)}.
\end{equation}
\end{lemma}

\begin{proof} Without loss of generality, up to
a translation and a rotation, which leave
our problem invariant, we may take~$x_o=0$
and the tangent hyperplane of~$\partial F$ at~$0$
to be~$\{x_n=0\}$. In this way, we write~$\partial F$
as the graph~$x_n=g(x')$, for~$|x'|\le 
1/\sqrt{n}$,
with~$\nabla g(0)=0$ and~${\mathcal{H}}(0)=\Delta g(0)$.
Up to a rotation of the horizontal coordinates, we also
suppose that~$D^2 g(0)$ is diagonal, with
eigenvalues~$\lambda_1,\dots,\lambda_{n-1}$.
In this way 
$$ g(y')=\frac{1}{2}\sum_{i=1}^{n-1}\lambda_i y_i^2 +h(y'),$$
and~$|h(y')|\le M|y'|^{2+\alpha}$.
So, for any~$|y'|\le r$,
\begin{equation}\label{in trap}
|g(y')|\le Mr^2\le \frac{r}{2},\end{equation}
thanks to~\eqref{r definition}.
We observe that, by rotational symmetry,
$$ \int_{\{|y'|\le r\}} y_j^2\,|y'|^{-(n+s)}\,dy'=
\int_{\{|y'|\le r\}} y_1^2\,|y'|^{-(n+s)}\,dy'$$
for any~$j=1,\dots,n-1$ and therefore, by summing up in~$j$,
\begin{equation*}\begin{split}
& \frac{\varpi r^{1-s}}{1-s}=
\int_{\{|y'|\le r\}} |y'|^{2-(n+s)}\,dy'\\
&\qquad=(n-1) \int_{\{|y'|\le r\}} y_1^2\,|y'|^{-(n+s)}\,dy'
=(n-1) \int_{\{|y'|\le r\}} y_i^2\,|y'|^{-(n+s)}\,dy'\end{split}
\end{equation*}
for any~$i=1,\dots,n-1$.
Therefore
\begin{equation}\label{cure}
\int_{\{|y'|\le r\}} \sum_{i=1}^{n-1}\lambda_i
y_i^2\,|y'|^{-(n+s)}\,dy'=\frac{\varpi r^{1-s}{\mathcal{H}}(0)}{
(n-1)(1-s)}.
\end{equation}
Let now
$$ G_s(\tau):=\int_0^\tau \frac{dt}{(1+t^2)^{(n+s)/2}}.$$
We observe that~$G_s(0)=0$, $G'_s(0)=1$ and~$|G_s''(\tau)|
=(n+s)(1+\tau^2)^{-(n+s+2)/2} |\tau|\le (n+1)|\tau|$.
Therefore, a Taylor expansion gives
$$ G_s(\tau)=\tau+\widetilde G_s(\tau),$$
with~$|\widetilde G_s(\tau)|\le C|\tau|^3$.
Therefore, if we write
$$ \widetilde g(y'):=
\frac{g(y')}{|y'|}= \frac{1}{2|y'|}\sum_{i=1}^{n-1} \lambda_i y_i^2
+\widetilde h(y')$$
with~$|\widetilde h(y')|=|h(y')|/|y'|\le
M|y'|^{1+\alpha}$, we have that
\begin{equation*}
\begin{split}
G_s(\widetilde g(y'))&=\widetilde g(y')+\widetilde G_s
(\widetilde g(y'))
\\ &=\frac{1}{2|y'|}\sum_{i=1}^{n-1} \lambda_i y_i^2
+\widetilde h(y')
+\widetilde G_s(\widetilde g(y'))
\\ &=\frac{1}{2|y'|}\sum_{i=1}^{n-1} \lambda_i y_i^2
+\ell(y'),
\end{split}\end{equation*}
with
$$ |\ell(y')|\le |\widetilde h(y')|+C|\widetilde g(y')|^3\le 
CM(|y'|^{1+\alpha}+|y'|^3)\le CM|y'|^{1+\alpha}$$
for any~$|y'|\le r$.
As a consequence of this and~\eqref{cure}, 
\begin{equation}\label{babel1}
\int_{\{ |y'|\le r\}} \frac{G_s(\widetilde g(y'))}{|y'|^{n+s-1}}\,dy'
=\frac{\varpi r^{1-s}{\mathcal{H}} (0)}{2(n-1)(1-s)}+\varepsilon_1
\end{equation}
with~$|\varepsilon_1|\le CM r^{1+\alpha-s}/(1+\alpha-s)\le CM$.
Now, since the map~$(0,+\infty)\ni
t\mapsto 1-e^{-t}$ is concave, we have that~$1-e^{-t}
\in[0,t]$, hence
$$ 1-r^{1-s}\in \big[ 0,(1-s)\log r^{-1}\big].$$
Accordingly, we may write~\eqref{babel1} as
\begin{equation}\label{Bab2}
\int_{\{ |y'|\le r\}} \frac{G_s(
\widetilde g(y'))}{|y'|^{n+s-1}}\,dy'
=\frac{\varpi {\mathcal{H}} (0)}{2(n-1)(1-s)}+\varepsilon_2
\end{equation}
with~$|\varepsilon_2|\le CM(1+\log r^{-1})$.

Now, we recall~\eqref{in trap}, we integrate in the
vertical coordinate and we substitute~$t:=y_n/|y'|$
to obtain that
\begin{eqnarray*}
&&\int_{K_r}\frac{\chi_F(y)-\chi_{\CC F}(y)}{|y|^{n+s}}
\,dy
\\ &=&\int_{|y'|\le r} \left[
\int_{-r}^{g(y')}\frac{dy_n}{(|y'|^2+|y_n|^2)^{(n+s)/2}}
-\int_{g(y')}^r\frac{dy_n}{(|y'|^2+|y_n|^2)^{(n+s)/2}}
\right]\,dy'
\\ &=& \int_{|y'|\le r} \frac{1}{|y'|^{n+s}}\left[
\int_{-r}^{g(y')}\frac{dy_n}{(1+(|y_n|/|y'|)^2)^{(n+s)/2}}
-\int_{g(y')}^r\frac{dy_n}{(1+(|y'|/|y_n|)^2)^{(n+s)/2}}
\right]\,dy'
\\ &=&\int_{|y'|\le r} \frac{1}{|y'|^{n+s-1}}\left[
\int_{-r/|y'|}^{\widetilde 
g(y')}\frac{dt}{(1+t^2)^{(n+s)/2}}
-\int^{r/|y'|}_{\widetilde 
g(y')}\frac{dt}{(1+t^2)^{(n+s)/2}}
\right]\,dy'
\\ &=&\int_{|y'|\le r} \frac{1}{|y'|^{n+s-1}}\left[
G_s(\widetilde g(y'))-G_s(-r/|y'|)
-G_s({r/|y'|})+G_s({\widetilde
g(y')}) \right]\,dy'.
\end{eqnarray*}
Therefore, since~$G_s$ is odd,
\begin{equation}\label{78deii3ejjjjejej}
\int_{K_r}\frac{\chi_F(y)-\chi_{\CC F}(y)}{|y|^{n+s}}
\,dy=2\int_{|y'|\le r} \frac{G_s(\widetilde g(y'))}{|y'|^{n+s-1}}
\,dy'=
\frac{\varpi {\mathcal{H}} (0)}{(n-1)(1-s)}+\varepsilon_3
\end{equation}
with~$|\varepsilon_3|\le CM(1+\log r^{-1})$, due to~\eqref{Bab2}.

Now, we point out the following cancellation:
\begin{eqnarray*}
&& \left|\int_{K_r\setminus B_r}
\frac{\chi_F(y)-\chi_{\CC F}(y)}{|y|^{n+s}}
\,dy\right|\le\int_{(K_r\setminus K_{r/\sqrt n}
)\cap \{|y_n|\le M|y'|\}}
\frac{1}{|y|^{n+s}}
\,dy
\\ &&\qquad\le 
CM\int_{r/\sqrt n}^r \rho^{-1-s}\,ds
=\frac{CM(n^{s/2}-1)}{s r^s}\le \frac{CM}{r}.
\end{eqnarray*}
Accordingly, we can write~\eqref{78deii3ejjjjejej} as
$$ \int_{B_r}\frac{\chi_F(y)-\chi_{\CC F}(y)}{|y|^{n+s}}
\,dy=
\frac{\varpi {\mathcal{H}} (0)}{(n-1)(1-s)}+\varepsilon_4
$$
with~$|\varepsilon_4|\le CM(1+\log r^{-1}+r^{-1})\le CMr^{-1}$.
This proves~\eqref{LAC1}.

Then,~\eqref{LAC3} follows from~\eqref{LAC1} and~\eqref{r definition},
by observing that, if~$M\in(0,1]$, we have that~$r=1/n$ so it does not 
depend on~$M$.
\end{proof}

\subsection{Construction of an
auxiliary function}\label{S COM}

The purpose of this section is
to obtain a special function, which is positive
in a large ball, and that
satisfies the correct inequality with respect
to the integral operator of~\eqref{the eq}
in a smaller ball.
This is needed to apply an appropriate variation of the
local Alexandrov-Bakelman-Pucci theory
of~\cite{Cabre, CS}, in order to localize the set
in which the solution we are considering becomes
positive. Indeed, the following function
is the one that replaces the 
auxiliary functions
in Lemma~4.1 of~\cite{Cabre} and
Corollary~9.3 of~\cite{CS} for our framework
(here, some technical complications
also arise since the operator in~\eqref{6gghhssh}
is both nonlocal and nonlinear in its dependence on the sets):

\begin{lemma}\label{barrier} Fix~$R>0$
and constants~$c_1,\dots,c_5>0$. Fix also~$c_0\in (0,c_1)$.
There exists~$C\ge1$ (possibly depending on~$c_0,\dots,c_5>0$
but independent of~$R$)
such that,
if~$1-s$, $\varepsilon\in(0,1/C]$, the following results hold.

There exists~$\Phi\in
C^\infty(\R^{n-1},[-C\varepsilon R,C\varepsilon R])$
satisfying the following conditions:
\begin{equation}\label{901212}
\begin{split}
&{\mbox{$\Phi(x')>\varepsilon R$ if~$|x'|\ge (c_1+c_2) R$, 
$\Phi(x')\le-4\varepsilon R$ 
if~$|x'|\le c_1 R$, and}}\\ &\qquad
\sup_{\R^{n-1}}|\nabla\Phi|+
R\,|D^2\Phi|\le {C\varepsilon}. \end{split}\end{equation}
Also, let~$L$ be an affine function with
\begin{equation}\label{D L}
|\nabla L|\le\frac{1}{C},
\end{equation}
set
\begin{equation}\label{789sd90djwwjjdfff1}{\mbox{
$\tilde\Phi:=L-\Phi$ and~$F:=\{ x_n<\tilde\Phi(x')\}$.}}
\end{equation} Then
\begin{equation}\label{6gghhssh}
(1-s) \int_{B_{c_3 R}(x)}\frac{\chi_{F}(y)-
\chi_{\CC F}(y)}{|x-y|^{n+s}}\,dy\ge\frac{c_4\varepsilon}{R^s}
\end{equation}
for any~$x\in \partial F\cap \{c_0 R < |x'|\le (c_1+c_2+c_5)R\}$.
\end{lemma}

\begin{proof} Up to replacing~$\Phi(x')$ with~$R\Phi(x'/R)$,
we may and do consider just the 
case~$R=1$.
Then, the function we will construct is depicted in Figure~2.

\begin{figure}[htbp]
\begin{center}
\resizebox{11.2cm}{!}{\input{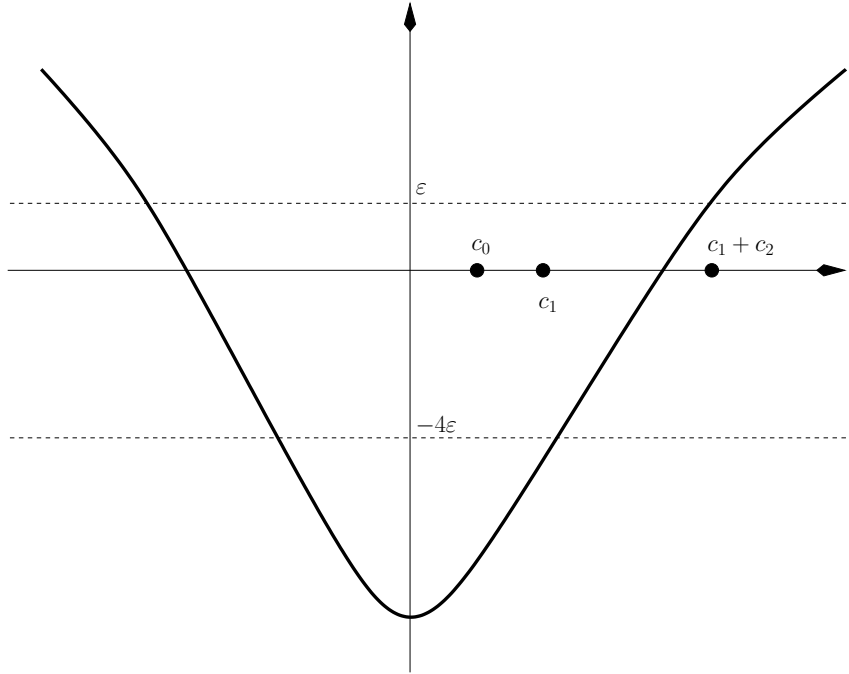}}
{\caption{\it The auxiliary function~$\Phi$ (with~$R=1$).}}
\end{center}
\end{figure}

More explicitly, we 
take~$\Phi$ to be smooth, radial,
radially increasing, satisfying~\eqref{901212}
with~$R=1$, and in fact
$$\|\Phi\|_{C^{2,\alpha}(\R^{n-1})}
\le C(1+\mu_q)\varepsilon,$$ and such 
that
\begin{equation*}
\Phi(x')=\varepsilon\left(\frac{c_0^q\mu_q}{c_1^q}-4-\frac{c_0^q\mu_q}{|x'|^q}\right)
\end{equation*}
if~$|x'|>c_0$. Here, $q>n-3$ is a fixed free parameter and~$\mu_q>0$
will be chosen appropriately large\footnote{At the moment we only need that~$\mu_q$
is so large that
$$ c_0^q \mu_q \left(\frac{1}{c_1^q}-\frac{1}{(c_1+c_2)^q}\right)\ge6.$$
In this way, if~$|x'|\ge c_1+c_2$, then
$$ \Phi(x')\ge\varepsilon\left( \frac{c_0^q\mu_q}{c_1^q}-4-\frac{c_0^q\mu_q}{(c_1+c_2)^q}\right)\ge2\varepsilon$$
that gives~\eqref{901212}.}
at the end of the proof.
We observe that, if~$|x'|>c_0$,
\begin{eqnarray*}
|\partial_i\Phi|&\le& {\varepsilon q\mu_q}c_0^q\,|x'|^{-q-1},\\
|\partial^2_{ij}\Phi|&\le&{\varepsilon q (q+3)\mu_q
}c_0^q\,|x'|^{-q-2}
\\{\mbox{and }}\;
-\Delta\tilde\Phi=\Delta\Phi &=&-{\varepsilon q(q-n+3)\mu_q}\,
c_0^q|x'|^{-q-2}.
\end{eqnarray*}
Accordingly, if~$|x'|>c_0$,
\begin{eqnarray*} && \sqrt{1+|\nabla\tilde\Phi|^2}\in [1,2]\\
{\mbox{and }}&&
\Delta\tilde\Phi-\left|
\frac{(D^2\tilde\Phi\nabla\tilde\Phi)\cdot\nabla
\tilde\Phi}{1+|\nabla\tilde\Phi|^2}
\right|\ge \frac{\varepsilon q(q+3-n)\mu_q}{4}c_0^q\,|x'|^{-q-2}
\end{eqnarray*}
as long as~$\varepsilon$ is small enough, thanks to~\eqref{D L}.
Hence, we estimate the mean curvature of~$\partial F$ at some
point~$x$ with~$|x'|\in (c_0,c_1+c_2+c_5]$ as 
$$ {\mathcal{H}} (x) = \frac1{\sqrt{1+|\nabla\tilde\Phi|^2}}\left(
\Delta\tilde\Phi-\frac{(D^2\tilde\Phi
\nabla\tilde\Phi)\cdot\nabla\tilde\Phi}{1+|\nabla\tilde\Phi|^2}
\right)
\ge \frac{\varepsilon \mu_q}{C} .$$
Therefore, if~$x\in\partial F$, $|x'|\in(c_0,c_1+c_2+c_5]$,
we have that
$$ (1-s) \int_{B_{c_3}(x)}\frac{\chi_F(y)-
\chi_{\CC F}(y)}{|x-y|^{n+s}}\,dy\ge 
\frac{\varepsilon \mu_q}{C^2}
-C(1+\mu_q)\varepsilon (1-s)\ge
\frac{\varepsilon \mu_q}{C^3} $$
thanks to~\eqref{LAC3}
in Lemma~\ref{limit curvature}, as long as~$1-s$ and~$\varepsilon$
are small enough. This and a suitably large choice of~$\mu_q$
give~\eqref{6gghhssh} (namely, we take~$\mu_q /C^3\ge c_4$).
\end{proof}

\subsection{Measure estimates for the oscillation}

We obtain the following measure estimate.
Such result may be
seen as the counterpart, in our framework,
of the measure estimate in Lemma~4.5 of~\cite{Cabre}
and Lemmata~8.6 and~10.1 of~\cite{CS}.

\begin{lemma}\label{SILV 8.6} Fix~$\overline C\ge1$.
Let~$\kappa\in \R$ and~$R>0$. 
Let~$u:\R^{n-1}\rightarrow\R$ be a Lipschitz function, with
\begin{equation}\label{AL}|\nabla u (x')|\le \overline C
\end{equation}
a.e.~$|x'|\le 3R$, and
\begin{equation}\label{USE}
u(x')\ge\kappa {\mbox{ for any }}|x'|\ge R.
\end{equation}
Let~$E:=\{ x_n<u(x')\}$.

Assume that, for any~$x\in \partial E\cap B_{4n}$,
\begin{equation}\label{sd77ef12345d}
(1-s)\int_{B_{R}(x)}\frac{\chi_E(y)-\chi_{\CC E}(y)}{|
x-y|^{n+s}}\,dy
\le \frac{\varepsilon}{R^s}.\end{equation}
Then, if
\begin{equation}\label{901212-bis}
\inf_{Q_{3R}} u\le \kappa+\varepsilon R
\end{equation}
we have that
\begin{equation}\label{step zero}
\Big|\big\{ u-\kappa \le M\varepsilon R
\big\}\cap Q_R \big\}\Big|\ge \mu R^{n-1},
\end{equation}
for appropriate universal constants~$M>1$
and~$\mu\in(0,1)$, as long as~$1-s$ and~$\varepsilon\in (0,1/C]$,
with~$C\ge1$ suitably large.

Here,~$M$, $\mu$ and~$C$ only depend on~$n$
and~$\overline C$.
\end{lemma}

\begin{proof} Up to translation, we may suppose 
that~$\kappa=0$.
Let~$\Phi$ be as in
Lemma~\ref{barrier} (with~$c_0,\dots,c_5$ to
be conveniently chosen in what follows).
Let~$v:=u+\Phi$ and~$\Gamma:\R^{n-1}\rightarrow\R$
be the convex envelope of~$v^-:=\min\{v,0\}$
in~$B_{6\sqrt n R}$, that is
$$ \Gamma(x):=
\left\{
\begin{matrix}
\displaystyle\sup_\Xi \ell(x) & {\mbox{ if }} |x'|<{6\sqrt n R}, \\
0 & {\mbox{ if }}|x'|\ge{6\sqrt n R}, 
\end{matrix}
\right.$$
where~$\Xi$ above is a short-hand notation
for all the affine functions~$\ell$
such that $\ell(y')\le v^-(y')$ for any
$|y'|<{6\sqrt n R}$ (see
pages 23--27 of~\cite{Cabre}
for the basic properties of the convex envelope).
Let~${\mathcal{T}}$ be the touching set
between~$v$ and~$\Gamma$, i.e.
$$ {\mathcal{T}}:=\{ x'\in\R^{n-1} {\mbox{ s.t. }}
\Gamma(x')=v(x')\}.$$
Let
$$ m_o:=-\inf_{Q_{3R}} v.$$
Notice that~$v\le u-4\varepsilon R$ in~$Q_{3R}$,
due to~\eqref{901212} (for this we choose~$c_1:=3\sqrt{n}/2$
in Lemma~\ref{barrier},
so that~$Q_{3R}\subseteq \{|x'|\le c_1 R\}$; the other constants~$c_0$, $c_2,\dots,c_5$
will be fixed in the sequel).

Therefore, by~\eqref{901212-bis},
$$ \inf_{Q_{3R}} v\le -2\varepsilon R,$$
so~$m_o\ge 2\varepsilon R$.

We recall that
all the hyperplanes with slope bounded by~$
m_o/(CR)$ belong to~$\nabla\Gamma(B_{6\sqrt n R})$
(see page~24 of~\cite{Cabre} and also~(3.9) there),
hence
\begin{equation}\label{Slope}
\varepsilon^{n-1}\le C\left(\frac{m_o}{R}\right)^{n-1}\le
C|\nabla\Gamma({\mathcal{T}})|.
\end{equation}
Now, for any~$\bar x'\in {\mathcal{T}}$,
we let
\begin{eqnarray*}
&& L(x'):=v(\bar x')+
\nabla\Gamma(\bar x')\cdot (x'-\bar x')\\{\mbox{and }}&&
\PPara:=L-\Phi.\end{eqnarray*}
We point out
that~$v>0$ in~$\{|x'|\ge 3\sqrt{n} R\}$, thanks to~\eqref{901212}
and~\eqref{USE} 
(for this, we choose~$c_2:=3\sqrt{n}/2$
in Lemma~\ref{barrier},
so that~$c_1+c_2:=3\sqrt{n}$).

In particular, since~$\Gamma\le0$,
we see that~$\bar x'\in{\mathcal{T}}\subseteq\{|x'|\le 3\sqrt{n}R\}$.

Also, from~\eqref{901212}, we have
\begin{equation}\label{par1}
|D^2\PPara|=|D^2\Phi|\le \frac{C\varepsilon}{R}.
\end{equation}
Moreover, $v$ is above~$\Gamma$ which is above~$L$ in~$B_{{6\sqrt n R}}$,
by convexity,
therefore, for any~$e\in S^{n-1}$
$$ 0\ge \Gamma(\bar x'+Re)\ge L(\bar x'+Re)
=v(\bar x')+R\nabla\Gamma(\bar x')\cdot e
\ge -C\varepsilon R+R\nabla\Gamma(\bar x')\cdot e$$
that is~$\nabla\Gamma(\bar x')\cdot e\le C\varepsilon$.
So, since~$e$ is an arbitrary unit vector, we get that
\begin{equation}\label{D L 2}
|\nabla L|=
|\nabla\Gamma(\bar x')|\le
C\varepsilon,\end{equation} and so, by~\eqref{901212},
\begin{equation}\label{par2}
|D\PPara|\le C\varepsilon.
\end{equation}
Now we observe that
\begin{equation}\label{bar trap}
{\mathcal{T}}\subseteq Q_R.
\end{equation}
The proof is by contradiction: if not,~$u+\Phi\ge L$
in~$\{ |x'|\le 6\sqrt{n}R\}$, with
equality at some~$\bar x'$ with~$\bar x'\not\in Q_R$.
In particular,~$|\bar x'|\ge R/2$.
Then, we can use Lemma~\ref{barrier}, with~$F$
as in~\eqref{789sd90djwwjjdfff1}
(notice that~\eqref{D L} is satisfied here due to~\eqref{D L 2}).
For this, we set~$\bar x:=(\bar x',u(\bar x'))\in\partial F$,
and we 
choose~$c_0:= 1/4$, $c_4:=2$ and~$c_5:= 
100\sqrt{n}$ in Lemma~\ref{barrier}.
In this way
since~$E\cap B_{{6\sqrt n R}}\supseteq F\cap B_{{6\sqrt n R}}$,
we deduce from~\eqref{6gghhssh} that
\begin{eqnarray*}
(1-s) \int_{B_R(\bar x)}\frac{\chi_{E}(y)-
\chi_{\CC E}(y)}{|x-y|^{n+s}}\,dy\ge 
(1-s) \int_{B_R(\bar x)}\frac{\chi_{F}(y)-
\chi_{\CC F}(y)}{|x-y|^{n+s}}\,dy\ge \frac{2\varepsilon}{R^s}
.\end{eqnarray*}
This is in contradiction with~\eqref{sd77ef12345d}
and so it establishes~\eqref{bar trap}.

Also, given~$\bar x'
\in {\mathcal{T}}$, we have that~$\PPara(\bar x')=v(\bar x')
-\Phi(\bar x')=u(\bar x')$ and
$$ \PPara\le\Gamma-\Phi\le v-\Phi=u.$$
This,~\eqref{AL}, \eqref{sd77ef12345d}, \eqref{par1} 
and~\eqref{par2}
say that the hypotheses of Lemma~\ref{S LEM} are fulfilled (up
to scaling~$\varepsilon$ to~$C\varepsilon$).
As a consequence, by~\eqref{star3}, for any~$M$ large enough,
\begin{equation}\label{9099}\frac{ \left| S^{(\bar x')}
\cap \Big\{
u(x')-u(\bar x')-\nabla\PPara(\bar x')\cdot (x'-\bar x')>
\displaystyle\frac{
M\varepsilon r_{\bar x'}^2}{R}
\Big\}\right|}{\big| S^{(\bar x')}\big|} \,\le\,
\frac{C}{M} \end{equation}
for a suitable ring~$S^{(\bar x')}:= \big\{ |x'-\bar x'|\in\big(
{r_{\bar x'}}/C, {r_{\bar x'}}\big)\big\}$ and
a suitable~$ {r_{\bar x'}}
\in (0,R]$.

On the other hand, by~\eqref{901212},
$$-\Phi(x')+\Phi(\bar x')+\nabla\Phi(\bar x')\cdot
(x'-\bar x')\ge -\frac{\varepsilon r_{\bar x'}^2}{R}\ge
-\frac{M\varepsilon r_{\bar x'}^2}{2R}$$
if~$x'\in  S^{(\bar x')}$, as long as~$M$ is big enough.
Consequently, using that~$v$ lies above~$\Gamma$ and
that~$\bar x'\in{\mathcal{T}}$, we have that
\begin{eqnarray*}
&& \Gamma(x')-\Gamma(\bar x')-\nabla\Gamma(\bar x')\cdot(x'-\bar x')
-\frac{M \varepsilon r_{\bar x'}^2}{2R}\\
&\le&\Gamma(x')-\Phi(x')-\Gamma(\bar x')+
\Phi(\bar x')
-\Big(\nabla\Gamma(\bar x')-
\nabla\Phi(\bar x')\Big)\cdot(x'-\bar x')
\\ &\le&
v(x')-\Phi(x')-v(\bar x')+
\Phi(\bar x')
-\Big(\nabla\Gamma(\bar x')-
\nabla\Phi(\bar x')\Big)\cdot(x'-\bar x')
\\ &=&
u(x')-u(\bar x')
-\nabla\PPara(\bar x')\cdot(x'-\bar x').
\end{eqnarray*}
The latter estimate and~\eqref{9099}
imply that
$$ \frac{ \left| S^{(\bar x')}
\cap \Big\{
\Gamma(x')-\Gamma(\bar x')-\nabla\Gamma(\bar x')\cdot (x'-\bar x')>
\displaystyle\frac{
M\varepsilon r_{\bar x'}^2}{2R}
\Big\}\right|}{\big| S^{(\bar x')}\big|} \,\le\,
\frac{C}{M} .$$
So, by taking~$M$ appropriately large and using
Lemma~8.4 of~\cite{CS} we deduce that
\begin{equation}\label{b782221122m}
\Gamma(x')-\Gamma(\bar x')-\nabla\Gamma(\bar x')\cdot (x'-\bar x')
\le\frac{C\varepsilon r_{\bar x'}^2}{R}
\end{equation}
for any~$|x'-\bar x'|< r_{\bar x'}/2$.

In particular, for any~$|x'-\bar x'|< r_{\bar x'}/4$,
we set~$\rho:=r_{\bar x'}/4$,
we plug the point~$x'+\rho e$ inside~\eqref{b782221122m},
we use the convexity of~$\Gamma$ twice
and we obtain
\begin{eqnarray*}
\frac{ C\varepsilon \rho^2}{R}&\ge&
\Gamma(x'+\rho e)-\Gamma(\bar x')-
\nabla\Gamma(\bar x')\cdot (x'+\rho e-\bar x')
\\ &\ge& \Gamma(x')+\rho \nabla\Gamma(x')\cdot e\\
&&\qquad
-\Gamma(\bar x')-
\nabla\Gamma(\bar x')\cdot (x'+\rho e-\bar x')
\\ &\ge& \Gamma(\bar x')+\nabla\Gamma(\bar x')\cdot(x'-\bar x')
+\rho \nabla\Gamma(x')\cdot e\\
&&\qquad -\Gamma(\bar x')-
\nabla\Gamma(\bar x')\cdot (x'+\rho e-\bar x')
\\ &=&\rho \big(\nabla\Gamma(x')-
\nabla\Gamma(\bar x')\big)\cdot e.
\end{eqnarray*}
So, since~$e$ is an arbitrary unit vector, it follows
that
\begin{equation*}
|\nabla\Gamma(x')-\nabla\Gamma(\bar x')|\le 
\frac{C\varepsilon r_{\bar x'}}{R}
\end{equation*}
for any~$|x'-\bar x'|< r_{\bar x'}/4$,
that is: the~$(n-1)$-dimensional
ball of radius~$r_{\bar x'}/4$ centered at~$\bar x'$
(which we now call~$B^{(\bar x')}$)
is sent, via the map~$\nabla\Gamma$,
inside the~$(n-1)$-dimensional
ball of radius~${C\varepsilon r_{\bar x'}}/{R}$ centered at~$
\nabla\Gamma(\bar x')$ (we observe that the
latter is a ball smaller by
a scale factor~$C\varepsilon/R$,
and let us call~$\widetilde B^{(\bar x')}$
such a ball).

Now we cover~${\mathcal{T}}$
with a countable, finite overlapping system of these balls, say~$\big\{
B^{(j)}\big\}_{j\in \N}$.
By the previous observations,
this covering induces a covering of~$\nabla\Gamma({\mathcal{T}})$
made of balls~$\big\{\widetilde B^{(j)}\big\}_{j\in \N}$, with~$|\widetilde B^{(j)}|
\le C(\varepsilon/R)^{n-1}|B^{(j)}|$.
So, we obtain
the measure estimate
\begin{equation}\label{0shhh65www}
|\nabla \Gamma({\mathcal{T}})|\le \sum_{j\in\N}
|\widetilde B^{(j)}|
\le C\left(\frac\varepsilon{R}\right)^{n-1}
\sum_{j\in\N} |B^{(j)}|.
\end{equation}
On the other hand,
we observe that, if~$ |x'-\bar x'|\le{r_{\bar x'}}$, 
then
\begin{eqnarray*}
u(x')&\le& u(x')-\Gamma(x') \\&\le&u(x')
-\Gamma(\bar x')-\nabla\Gamma(\bar
x')\cdot(x'-\bar x')
\\ &=& u(x')
-u(\bar x')-\Phi(\bar x')-\big(\nabla\PPara(\bar
x')
+\nabla\Phi(\bar x')\big)\cdot(x'-\bar x')
\\ &\le& u(x')-u(\bar x')-\Phi(x')-\nabla\PPara(\bar x')
\cdot(x'-\bar x')+\frac{C\varepsilon}{R}|x'-\bar x'|^2
\\ &\le&  u(x')-u(\bar x')-\nabla\PPara(\bar x')
\cdot(x'-\bar x')+C\varepsilon R
\end{eqnarray*}
thanks to the convexity of~$\Gamma$ and~\eqref{par1}.
Therefore
\begin{equation}\label{L.A.79}\begin{split}
& S^{(\bar x')}
\cap \Big\{
u(x')-u(\bar x')-\nabla\PPara(\bar x')\cdot (x'-\bar x')\le
\displaystyle\frac{
M\varepsilon r_{\bar x'}^2}{R}\Big\}\\
\subseteq\;\,&
S^{(\bar x')}
\cap \{u(x')\le C\varepsilon R\}\\
\subseteq\;\,&
B^{(\bar x')}
\cap \{u(x')\le C\varepsilon R\}.
\end{split}\end{equation}
Also, by~\eqref{9099}
\begin{eqnarray*}&&\left| S^{(\bar x')}
\cap \Big\{
u(x')-u(\bar x')-\nabla\PPara(\bar x')\cdot (x'-\bar x')\le
\displaystyle\frac{
M\varepsilon r_{\bar x'}^2}{R}
\Big\}\right|\\ &&\qquad\ge \left(1-
\frac{C}{M}\right)\,| S^{(\bar x')}|\ge
\frac{| S^{(\bar x')}|}2
\ge \frac{| B^{(\bar x')}|}{C}.\end{eqnarray*}
This and~\eqref{L.A.79} give that
$$ | B^{(\bar x')}|\le
C |B^{(\bar x')}
\cap \{u(x')\le C\varepsilon R\}|.$$
Gathering this estimate,
\eqref{Slope} and \eqref{0shhh65www},
and using the finite overlapping property of~$\big\{
B^{(j)}\big\}_{j\in \N}$, we conclude that
\begin{equation}\label{step zero a}
\begin{split}
& \varepsilon^{n-1}\le 
C|\nabla\Gamma({\mathcal{T}})|
\le C\left( \frac\varepsilon{R}\right)^{n-1}
\sum_{j\in\N} |B^{(j)}|
\\ &\qquad \le C\left( \frac\varepsilon{R}\right)^{n-1}
\sum_{j\in\N} 
\Big|B^{(j)}
\cap \{u\le C\varepsilon R\}\Big|\le C\left( 
\frac\varepsilon{R}\right)^{n-1}
\left|
\bigcup_{j\in\N} B^{(j)}
\cap \{u\le C\varepsilon R\}\right|
.\end{split}\end{equation}
Accordingly,~\eqref{step zero}
is a consequence of~\eqref{step zero a}
and~\eqref{bar trap}.
\end{proof}

\subsection{Uniform improvement
of flatness}

The cornerstone of the regularity theory of~\cite{CRS}
is Lemma~6.9 there, to wit a Harnack Inequality, according to which
$s$-minimal surfaces become more and more flat
when we get
closer and closer to any of their points.
However, the estimates in Lemma~6.9
of~\cite{CRS} are all uniform
when~$s$ is bounded away from
both~$0$ and~$1$, but they do degenerate
as~$s\rightarrow1^-$ (see, in particular, the estimate on~$I_1$
on page~1129 of~\cite{CRS}), therefore such result
cannot be applied directly in our framework.

For this scope, we provide the following
result, which is a version of Lemma~6.9
of~\cite{CRS} with uniform estimates as~$s\rightarrow1^-$.
In fact, the 
reader may compare Lemma~\ref{ovidiu 6.9}
here below with Lemma~6.9 in~\cite{CRS}: the only
difference is that the estimates here are
uniform as~$s\rightarrow1^-$.

Our proof is completely different from the one in~\cite{CRS}
and it is based
on the uniformity of the results obtained
in the preceding sections, together with
a Calder{\'o}n--Zygmund iteration, which needs
to distinguish between two scales of the dyadic cubes.

\begin{lemma}\label{ovidiu 6.9}
Fix~$s_o\in(0,1)$ and~$\alpha\in(0,1)$. Then, there
exist~$K\in \N$
and~$d\in(0,1)$ which only depend
on~$n$, $\alpha$ and~$s_o$, for which the following result holds.

Let~$a:=2^{-K\alpha}$. Let~$E$ be a set with~$s$-minimal perimeter
in~$B_{2^{K+1}}$, with~$s\in[1/10,1)$. Assume that
\begin{equation}\label{X51}
\partial E\cap B_1\subseteq \{ |x_n|\le a\}
\end{equation}
and, for any~$i\in \{0,\dots,K\}$,
\begin{equation}\label{X52}
\partial E\cap B_{2^i}\subseteq \{ |x\cdot
\nu_i|\le a2^{i(1+\alpha)}\}
\end{equation}
for some~$\nu_i\in{\rm S}^{n-1}$.
Then 
\begin{equation}\label{X53}\begin{split}
&{\mbox{either }}\,\partial E\cap B_d\subseteq \{ x_n\le 
a(1-d^2)\}\\ &{\mbox{or }}\,
\partial E\cap B_d\subseteq \{ x_n\ge a(-1+d^2)\}.\end{split}
\end{equation}
\end{lemma}

\begin{proof} The proof is not simple, but the naive idea is to 
argue by
contradiction, supposing that there is
a sequence of~$E_j$'s
that oscillate too much. Then one performs the following steps:
\begin{itemize}
\item By~\cite{CV}, one gets a sequence~$s_j\rightarrow 1^-$
for which~$E_j$ approaches a classical minimal surface~$E_\star$;
\item By~\eqref{CorCC},
one shadows~$E_j$ with level sets of distance 
functions~$u^\pm_j$ from above and below, and the graphs
of $u^\pm_j$ are close to~$\partial E_\star$ as~$s_j\rightarrow 1^-$;
\item Since (by contradiction) we assumed~$E_j$ to oscillate
too much, there are points of~$E_j$ (and so of the graphs of~$u^\pm_j$)
that stay very close to the bottom and the top of the cylinder
of height~$a$;
\item Accordingly, from the fact that there is a point
for which~$u^-_j$ is close to the bottom, we deduce
that~$u^-_j$ is close to the bottom in a rather large set:
for this, one 
needs to use
a dyadic cube argument -- when the cubes are reasonably big,
one can repeat Lemma~\ref{SILV 8.6}, and when the cubes get too small
one takes advantage of the regularity theory for the classical
minimal surface~$E_\star$;
\item Analogously, from the fact that there is a point
for which~$u^+_j$ is close to the top, we deduce
that~$u^+_j$ is close to the top in a rather large set;
\item In particular, we find a point for which~$u^+_j$ is
close to the top and~$u^-_j$ close to the bottom, that is~$u^+_j
-u^-_j$ is of the order of~$a$;
\item This is in contradiction with~\eqref{s8822211a}
and so it completes the proof.
\end{itemize}
We remark that, in these arguments,
there are two
uncorrelated scales involved.
One is the flatness of order one
(which, in the course of the proof,
will be dominated by a configuration of cylinders whose ratio
between the height and the base is
some~$\varepsilon^\star$); the other is 
the one
induced by
the criticality ratio for the minimal surfaces flatness condition
(which is some universal~$\varepsilon_o$).
Of course, both these configurations
are somewhat induced by the trapping of the surface
in a strip of small size~$a$.
The interplay between these two scales
is what allows us to choose 
the critical $s$ in an independent way,
and so to decouple the ratio of the scales involved.
Finally , this implies also that as
the flatness~$\varepsilon_\flat$ of~\eqref{main trap}
improves (while
the classical minimal surfaces flatness~$\varepsilon_o$
is a fixed constant), we can apply
the decrease of oscillation more and more times, so that in the vertical
blow up limit we get a H\"older graph, that is
harmonic in viscosity sense (see~\cite{CRS}).

Below is the full detail discussion.
The proof is by contradiction.
If the claim were false, 
since the estimates of Lemma~6.9 of~\cite{CRS} are
uniform when~$s\ge1/10$ is bounded away from~$1$,
it follows that
there exist
\begin{equation}\label{XoX5}
s_j\rightarrow 1^-,\end{equation}
and a sequence~$E_j$ of~$s_j$-minimal surfaces
in~$B_{2^{K+1}}$ such that
\begin{equation}\label{trap 0}
\partial E_j\cap B_1\subseteq \{ |x_n|\le a\}
\end{equation}
and, for any~$i\in \{0,\dots,K\}$,
\begin{equation}\label{trap i}
\partial E_j\cap B_{2^i}\subseteq \{
|x\cdot \nu_i|\le a2^{i(1+\alpha)}\}.
\end{equation}
for suitable~$\nu_i\in{\rm S}^{n-1}$, but
\begin{equation}\label{not trap}
\partial E_j\cap B_d\cap \{ x_n\ge a(1-d^2)\}\ne\varnothing
{\mbox{ and }}\;
\partial E_j\cap B_d\subseteq \{ x_n\le a(-1+d^2)\}
\ne\varnothing.\end{equation}
By~\eqref{XoX5} and
Theorem~7 in~\cite{CV}, we have that~$\chi_{E_j}$
converges in~$L^1(B_{(9/7) 2^K})$ to 
some~$E_\star$ (possibly up to subsequence).
Therefore (see the Remark after Corollary~17
in~\cite{CV}) $E_j$ approaches~$E_\star$ uniformly
in~$B_{(8/7) 2^K}$ and then, by Theorem~6 in~\cite{CV},
we have that~$E_\star$ is a classical minimal surface
in~$B_{2^K}$.

We will define~$\gamma_j$ to be the distance between~$E_j$
and~$E_\star$ in~$B_{2^K}$: by construction
\begin{equation}\label{dw110dfu2bxq0}
\lim_{j\rightarrow+\infty}\gamma_j=0.
\end{equation}
Let also
$$ \delta_j:=
a\gamma_j^{1/(1+\alpha)},$$
and notice that
\begin{equation}\label{dw110dfu2bxq}
\lim_{j\rightarrow+\infty}\delta_j=0.
\end{equation}
Now, we observe that~$K\alpha>4(1+\alpha)$ if~$K$
is large enough, and so we can
take~$K'\in \N$ such that
\begin{equation}\label{sd9790e3c-33yyy}
\frac{K\alpha}{2(1+\alpha)}-1<K'\le
\frac{K\alpha}{2(1+\alpha)}.\end{equation}
Now, we denote by~$\varepsilon_o$ the 
flattening constants of the classical minimal surfaces
(see, e.g.,~\cite{CC} and references therein) according to which
if a minimal surface is trapped in a
cylinder whose ratio between the height and the base
is below~$\varepsilon_o$, then the minimal surface
is a~$C^{1,\alpha}$-graph in half the cylinder.
By~\eqref{trap i}, \eqref{sd9790e3c-33yyy}
and the uniform convergence 
of~$E_j$, we see that, for large~$K$
(possibly in dependence of~$\varepsilon_o$),
\begin{eqnarray*}
&& \partial E_\star \cap B_{2^{K'}}\subseteq
\{ |x\cdot \nu_{K'}|\le 2^{-K\alpha} 2^{K'(1+\alpha)}\}
\\ &&\qquad\subseteq\{ |x\cdot \nu_{K'}|\le 2^{-K\alpha/2} 
\}\subseteq \{ |x\cdot \nu_{K'}|\le\varepsilon_o\},\end{eqnarray*}
and so
\begin{equation}\label{8d9e93333kk}
{\mbox{$\partial E_\star \cap B_{2^{K'-1}}$ is
a~$C^{1,\alpha}$-graph.}}
\end{equation}
Now, we use Corollary~\ref{CorCC}
with~$\gamma:=\gamma_j$ 
and~$\delta:=\delta_j$:
for this,
we define
\begin{equation}\label{dw110dfu2bxq.2}
{\mathcal{S}}^\pm_j:=\{ x\in \R^n {\mbox{ s.t. }} d_{E_j}(x)=
\pm \delta_j\}\end{equation}
and we deduce from~\eqref{8d9e93333kk} and
Corollary~\ref{CorCC} that~${\mathcal{S}}^\pm_j \cap B_{2^{K'-2}}$ is
\begin{equation}\label{e L}
\begin{split}
&{\mbox{ the graph of a uniformly Lipschitz function, say~$u^\pm_j$.}}
\end{split}\end{equation}
Also, from~\eqref{s8822211a}, \eqref{dw110dfu2bxq0}
and~\eqref{dw110dfu2bxq}, we have that
\begin{equation}\label{so long}
u^+_j(x')-u^-_j(x')\le C\delta_j
\end{equation}
for any~$|x'|\le 1$, as long as~$j$ is large enough.

Now we will concentrate on~$u^-_j$ (the case of~$u^+_j$
being specular): we
set~$E^-_j:= \{x_n<u^-(x')\}$, so that~$\partial E^-_j
= {\mathcal{S}}^-_j$. {F}rom~\eqref{not trap}
and the fact that~${\mathcal{S}}^-_j$ lies
below~$E_j$, we obtain that there exists~$\zeta'\in\R^{n-1}$
with
\begin{equation}\label{ds73jjjjj11}
|\zeta'|\le d
\end{equation}
and
\begin{equation}\label{pr915}
u^-_j(\zeta')\le a(-1+d^2).\end{equation}
As usual in these types of proofs, the convenient~$d$ in our argument will be chosen later on,
in dependence of the constants of the previous lemmata (see~\eqref{ne.2}
below).

Now, we use the following notation:
given any~$x\in\G^-_j$,
let~$y(x)\in \partial E_j$ such that~$|y(x)-x|=\delta_j$,
and let~$\nu(x):=y(x)-x$. Then
\begin{equation}\label{incl}
E^-_j +\nu(x)\subseteq \overline{E}.
\end{equation}
Indeed, if~$p\in E^-_j +\nu(x)$, we have that~$p-\nu(x)\in E^-_j$
and so~$\overline{B_{\delta_j}(p-\nu(x))}\subseteq \overline{E}_j$.
Then, since~$|\nu(x)|=\delta_j$, we have~$p\in
\overline{B_{\delta_j}(p-\nu(x))}\subseteq \overline{E}_j$, 
proving~\eqref{incl}.

Moreover~$\partial E$ has zero Lebesgue measure
(see, e.g., Corollary~4.4(i) of~\cite{CRS}), thus we
infer from~\eqref{incl} that, if~$x_o\in\partial E^- _j$,
\begin{equation}\label{Incl}
\chi_{E^- _j+\nu(x_o)}\le\chi_{E} \qquad{\mbox{ and }}\qquad
\chi_{\CC(E^-_j +\nu(x_o))}\ge\chi_{\CC E}.
\end{equation}
Therefore, using~\eqref{Incl}, the Euler-Lagrange equation
satisfied by~$E$ (see Theorem~5.1 of~\cite{CRS})
and the change of variable~$z:=x+\nu(x_o)$, we obtain
\begin{equation}\label{E.L.}
\begin{split}
&\int_{\R^n} 
\frac{\chi_{E^-_j}(x)-\chi_{\CC(E^-_j)}(x)}{|x-x_o|^{n+s_j}}\,dx
=
\int_{\R^n} \frac{\chi_{E^-_j
+\nu(x_o)}(z)-\chi_{\CC(E^-_j +\nu(x_o))}(z)}{|z
-y(x_o)|^{n+s_j}}\,dz
\\ &\qquad
\le \int_{\R^n} \frac{\chi_{E}(z)-\chi_{\CC E}
(z)}{|z-y(x_o)|^{n+s_j}}\,dz\le0
\end{split}\end{equation}
for any~$x_o\in\partial E^-_j\cap B_C$.
On the other hand, by~\eqref{trap i}, we have that~$|x_o\cdot\nu_i|\le
C a 2^{i(1+\alpha)}$, and so
\begin{eqnarray*}
&& \partial E_j\cap B_{2^i}(x_o)\subseteq
\partial E_j\cap B_{2^{i+C}}\\
&&\qquad\subseteq
\{|x\cdot\nu_i|\le C
a2^{i(1+\alpha)} \}
\subseteq
\{|(x-x_o)\cdot\nu_i|\le C
a2^{i(1+\alpha)} \}
\end{eqnarray*}
for any~$1\le i\le K-C$.
Therefore, for $j$ large,
$$ \partial E^-_j\cap B_{2^i}(x_o)\subseteq
\{|(x-x_o)\cdot\nu_i|\le
Ca2^{i(1+\alpha)} \}$$
for any~$1\le i\le K-C$.
As a consequence,
we obtain the
following cancellation:
\begin{equation}\label{8.20p}\begin{split}
& \left|\int_{\CC B_1(x_o)}
\frac{\chi_{E^-_j}(x)-\chi_{\CC(E^-_j)}(x)}{|x-x_o|^{n+s_j}}\,dx\right|
\\ &\le
\sum_{i=1}^{K-C}
\left|\int_{B_{2^i}(x_o)\setminus B_{2^{i-1}}(x_o)}
\frac{\chi_{E^-_j}(x)-\chi_{\CC(E^-_j)}(x)}{|x-x_o|^{n+s_j}}\,dx\right|
+
\left|\int_{\CC B_{2^{K-C}}(x_o)}
\frac{\chi_{E^-_j}(x)-\chi_{\CC(E^-_j)}(x)}{|x-x_o|^{n+s_j}}\,dx\right|
\\ &\le
C\left[ \sum_{i=1}^{K-C} 
\limits\int_{ {B_{2^i}(x_o)\setminus B_{2^{i-1}}(x_o)}\atop{
\{ |(x-x_o)\cdot\nu_i| \le Ca2^{i(1+\alpha)}\}
}}
\frac{1}{|x-x_o|^{n+s_j}}\,dx
+
\int_{\CC B_{2^{K-C}}(x_o)}
\frac{1}{|x-x_o|^{n+s_j}}\,dx\right]
\\ &\le C\left[ \sum_{i=1}^{K-C}
\int_{2^{i-1}}^{2^i} \frac{a2^{i(1+\alpha)} \rho^{n-2}}{\rho^{n+s_j}}
\,d\rho
+
\int_{2^{K-C}}^{+\infty}
\frac{\rho^{n-1}}{\rho^{n+s_j}}\,d\rho\right]
\\ &\le Ca
\end{split}\end{equation}
provided that~$j$ is big enough (in particular,~$s_j$
is larger than~$\alpha$).

Therefore, by~\eqref{E.L.}
and~\eqref{8.20p},
for any~$x_o\in \partial E^-_j\cap B_C$,
\begin{equation}\label{ne.1} \int_{B_1(x_o)}
\frac{\chi_{E^-_j}(x)-\chi_{\CC(E^-_j)}(x)}{|x-x_o|^{n+s_j}}\,dx
\le Ca .\end{equation}
With this, we 
are in position to
obtain a finer bound in measure,
often referred to with the name of~``$L^\beta$-estimate''
(see, e.g.,
Lemma~4.6 of~\cite{Cabre}
and Lemma~9.2 of~\cite{CS}
for the corresponding results for fully nonlinear
or fractional operators, the proof of which is based
on related, but quite different, techniques).
Such estimate will be based on a
Calder{\'o}n--Zygmund type dyadic cube decomposition.
According to the different scales involved,
we use either
a repeated version of Lemma~\ref{SILV 8.6} or
the vicinity of the classical minimal surface~$E_\star$
to deduce the necessary rigidity features.

Here are the details of such $L^\beta$-estimate.
We take~$\mu\in(0,1)$ and~$M\in(1,+\infty)$
as in Lemma~\ref{SILV 8.6}, and we fix a large integer~$k_o$
such that
\begin{equation}\label{k o}
(1-\mu)^{k_o}\le\frac14.
\end{equation}
Then, we choose
\begin{equation} \label{ne.2}
d:=\frac{1}{2M^{k_o}} \in(0,1),\end{equation}
we set~$a_j:=a+\delta_j+\gamma_j$,
and we claim that, for any $k\in\N$, with~$1\le k\le k_o$, we have that
\begin{equation}\label{L beta}
\left| \Big\{ u^-_j +a_j\ge
\frac{ a_j M^{k-k_o}}{2} \Big\} \cap Q_1\right|\le 
(1-\mu)^k
\end{equation}
as long as~$j$ is large enough.

Indeed, when~$k=1$,~\eqref{L beta} is a consequence
of~\eqref{step zero}, by applying
Lemma~\ref{SILV 8.6} here with~$\varepsilon:=d a_j$, $\kappa:=-a_j$
and~$R:=1$ -- for this recall~\eqref{pr915}, \eqref{ne.1}
and~\eqref{ne.2}
in order to check~\eqref{sd77ef12345d}
and~\eqref{901212-bis}, and consider
the complement set in~\eqref{step zero}:
such configuration is sketched in Figure~3.

\begin{figure}[htbp]
\begin{center}
\resizebox{11.2cm}{!}{\input{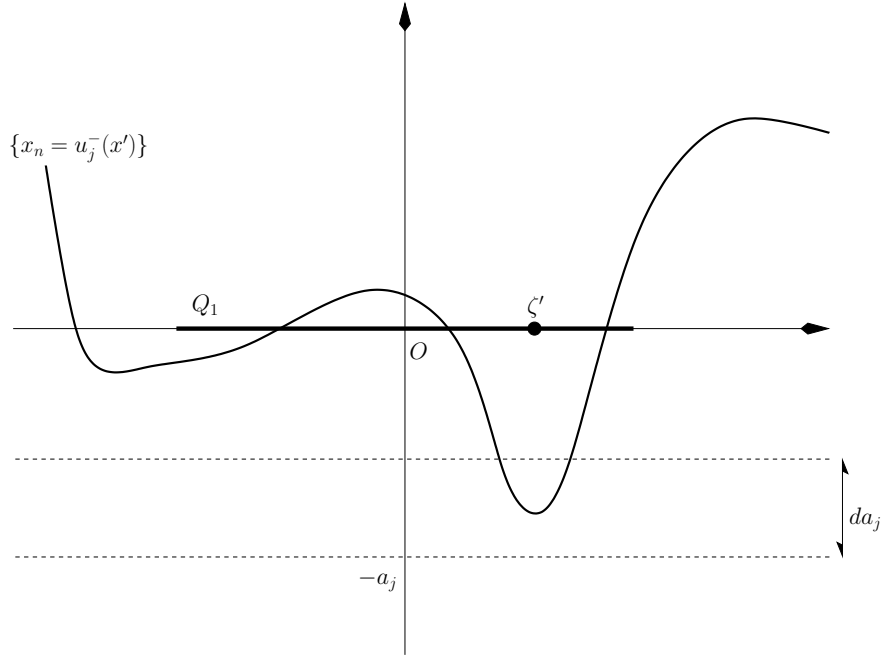}}
{\caption{\it Proving~\eqref{L beta} when $k=1$.}}
\end{center}
\end{figure}

Then, we proceed by induction, by supposing
that~\eqref{L beta} holds for~$k-1$, and we prove
it for~$k\le k_o$.
For simplicity, we just
perform the step from~$k=1$ to~$k=2$ (the others are
analogous).
For this, we define
$$ A:= \Big\{ u^-_j +a_j>
\frac{ a_j M^{2-k_o}}{2} \Big\} \cap  Q_1
\;{\mbox{ and }}\;
B:= \Big\{ u^-_j +a_j>
\frac{ a_j M^{1-k_o}}{2} \Big\} \cap Q_1 .$$
Notice that
\begin{equation}
A\subseteq B\subseteq Q_1
\end{equation} and
\begin{equation}
|A|\le \left|
\Big\{ u^-_j +a_j>
\frac{ a_j M^{1-k_o}}{2} \Big\} \cap Q_1\right|
\le 1-\mu,
\end{equation}
since we know that~\eqref{L beta} holds when~$k=1$.

Now we take a dyadic cube decomposition
of~$Q_1$, with the notation that
if~$Q$ is one of the cubes of the family, its predecessor
is denoted by~$\tilde Q$.
We claim that
\begin{equation}\label{post}
{\mbox{if $|A\cap Q|>(1-\mu) |Q|$ then
$\tilde Q\subseteq B$.
}}
\end{equation}
Notice that if~\eqref{post} holds, then, by Lemma~4.2
of~\cite{Cabre} (applied here with~$\delta:=1-\mu$)
and the inductive 
assumption (that is, in this case,~\eqref{L beta}
with~$k=1$), we have that
\begin{equation*}\begin{split}
& \left|\Big\{ u^-_j +a_j>
\frac{ a_j M^{2-k_o}}{2} \Big\} \cap Q_1\right|=|A|
\\ &\qquad\le(1-\mu) |B|=(1-\mu)
\left|  \Big\{ u^-_j +a_j>
\frac{ a_j M^{1-k_o}}{2} \Big\} \cap Q_1\right|
\le (1-\mu)^2.\end{split}
\end{equation*}
This would complete the induction necessary for
the proof of~\eqref{L beta}, hence we focus on the proof 
of~\eqref{post}.

For the proof of~\eqref{post}, we argue by
contradiction, by supposing that 
\begin{equation}\label{DQ}
|A\cap Q|>(1-\mu)|Q|
\end{equation}
but there exists~$\xi'\in\tilde Q\setminus B$, i.e.
\begin{equation}\label{xi prime}
u^-_j(\xi') +a_j\le
\frac{ a_j M^{1-k_o}}{2}.
\end{equation}
We denote by~$\ell$
the width of~$Q$
(which is, say, centered at some~$x_\star'\in\R^{n-1}$).
We need to distinguish two cases,
according to the scale of the cube~$Q$, namely,
we distinguish whether or 
not~$a_j/\ell \le 
\varepsilon^\star$, using either
Lemma~\ref{SILV 8.6} or the minimal surface rigidity
(here~$\varepsilon^\star$ is a small quantity,
say the minimum between the threshold for the classical
minimal surface regularity~$\varepsilon_o$, as
introduced after~\eqref{sd9790e3c-33yyy},
and the small constants given
by Lemma~\ref{SILV 8.6}: a precise requirement about this will
be taken after~\eqref{EL2}).

If
\begin{equation}\label{less}
a_j/\ell\le \varepsilon^\star,\end{equation} we use
Lemma~\ref{SILV 8.6}.
For this scope, given~$x_o\in\partial E^-_j\cap B_{C}$, we notice that
\begin{eqnarray*}
&& \left| \int_{B_1\setminus B_\ell(x_o)}
\frac{\chi_{E^-_j} (x)-\chi_{\CC (E^-_j)}(x) }{|x-x_o|^{n+s_j}}\,dx
\right|=
 \left| \int_{(B_1\setminus B_\ell(x_o))\cap \{|x_n|\le Ca_j\}}
\frac{\chi_{E^-_j} (x)-\chi_{\CC (E^-_j)}(x) }{|x-x_o|^{n+s_j}}\,dx
\right|\\
&&\qquad \le C \int_{(\CC B_\ell(x_o))\cap \{|x_n|\le Ca_j\}}
\frac{1}{|x'-x_o'|^{n+s_j}}\,dx \le C a_j \int_\ell^{+\infty}
\frac{\rho^{n-2}}{\rho^{n+s}}\,d\rho\le \frac{C a_j}{\ell^{1+s}}.
\end{eqnarray*}
As a consequence, recalling~\eqref{ne.1},
\begin{equation}\label{EL2}
(1-s_j) \int_{B_\ell(x_o)}
\frac{\chi_{E^-_j} (x)-\chi_{\CC (E^-_j)}(x) }{|x-x_o|^{n+s_j}}\,dx
\le \frac{C (1-s_j) a_j}{\ell^{1+s}}.
\end{equation}
With this, we are in position to apply
Lemma~\ref{SILV 8.6} with~$\kappa:=-a_j$,
$R:=\ell$ and~$\varepsilon:=a_j 
M^{1-k_o}/
(2\ell)$ -- notice indeed that~\eqref{sd77ef12345d}
follows from~\eqref{EL2}, \eqref{901212-bis}
follows from~\eqref{xi prime} and, recalling~\eqref{less},
we see that~$\varepsilon\le
\varepsilon^\star M^{1-k_o}/2$ which is small if so 
is~$\varepsilon^\star$: this configuration is
represented in Figure~4.

\begin{figure}[htbp]
\begin{center}
\resizebox{11.2cm}{!}{\input{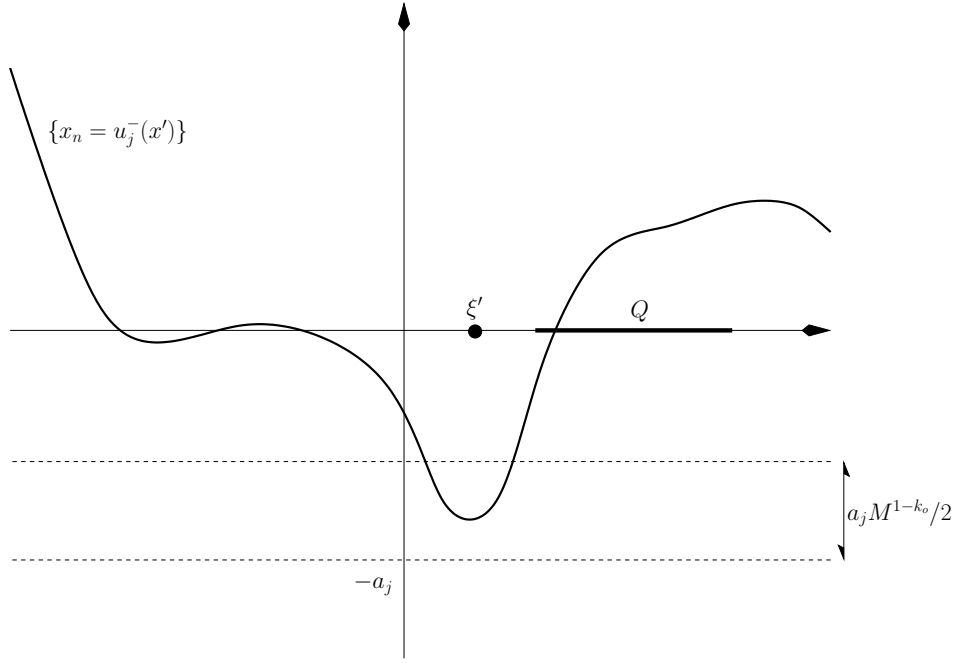}}
{\caption{\it Proving the inductive step of~\eqref{L beta}
when~$a_j/\ell\le \varepsilon^\star$.}}
\end{center}  
\end{figure} 

So, we obtain from~\eqref{step zero}
that
\begin{eqnarray*}
&& |A\cap Q|=\Big|\big\{ u^-_j +a_j > \frac{a_j M^{2-k_o}}{2}
\big\}\cap Q\Big|
\\ &&\qquad=
\Big|\big\{ u^-_j -\kappa> 
M\varepsilon R
\big\}\cap Q\Big|\le (1-\mu) |Q|,\end{eqnarray*}
which is in contradiction with~\eqref{DQ}.
This proves~\eqref{post} if~\eqref{less} holds true.

Now we deal with the case in which~$a_j/\ell \ge
\varepsilon^\star$, and we fix~$\theta\in(0,1)$
to be chosen suitably small in the sequel.
We set~$p:=a_j/(\theta^2 \varepsilon^\star)$.
Notice that, for small~$\theta$, we
have that~$p>10 a_j/\varepsilon^\star \ge10\ell$.
Also, the ratio between~$a_j$ and~$p$ is below~$\theta^2 
\varepsilon^\star$, 
hence a minimal surface 
that is trapped inside~$\{|x'|\le p\}\times\{|x_n|\le 8 a_j\}$
is the graph of a function~$\omega$, with~$|\nabla\omega|\le 
\theta^{3/2} \varepsilon^\star$. 
Accordingly,
\begin{equation}\label{omega os}\begin{split}&{\mbox{
the oscillation of~$\omega$ in~$\{|x'_i|\le 6\ell\}$}}\\
&\qquad{\mbox{
is bounded by~$\theta\varepsilon^\star\ell \le \theta a_j$.}}
\end{split}\end{equation}
Keeping this in mind,
we take~$j$ so large that~$\gamma_j$, i.e. the distance 
between~$E_j$
and~$E_\star$ is less than~$\theta^2\varepsilon^\star p/2$ 
(recall~\eqref{dw110dfu2bxq0}).
Also, for large~$j$, we have that the graph of~$u^-_j$ is
at distance~$\delta_j$ less than~$\theta^3\varepsilon^\star p/2$ 
from~$E_j$,
and so less than~$\theta^3\varepsilon^\star p$ from~$E_\star$
(recall~\eqref{dw110dfu2bxq}
and~\eqref{dw110dfu2bxq.2}).

Accordingly, $\partial E^\star\cap \{|x'_i|\le 6\ell\}$
is trapped in a slab of width~$4a_j+2\theta^3\varepsilon^\star p<8a_j$,
and, by~\eqref{xi prime}, its boundary contains a point
with vertical entry below~$
(a_j M^{1-k_o}/{2})+\theta^3\varepsilon^\star p$.
Then, by~\eqref{omega os}, the whole of~$\partial E^\star\cap \{|x'_i|\le 
4\ell\}$ has vertical entry below
$$ -a_j+(a_j M^{1-k_o}/{2})+\theta^3\varepsilon^\star p+\theta a_j.$$
Consequently, the graph of~$u^-$ on~$Q$ would stay below
\begin{eqnarray*}
&& -a_j+(a_j M^{1-k_o}/{2})+\theta^3\varepsilon^\star p+\theta a_j
+\theta^3\varepsilon^\star p
\\ &&\qquad=-a_j+(a_j M^{1-k_o}/{2})+3\theta a_j
< -a_j+(a_j M^{2-k_o}/{2}),
\end{eqnarray*}
as long as we choose~$\theta<M^{1-k_o} (M-1)/6$. Hence,~$A\cap Q=
\varnothing$,
which is in contradiction with~\eqref{DQ}.
This ends the proof of~\eqref{post},
and therefore the one of~\eqref{L beta}.

As a consequence, by taking~$k:=k_o$ in~\eqref{L beta}
and recalling~\eqref{k o}, we obtain that
\begin{equation}\label{L beta 1}
\left| \Big\{ u^-_j <
-\frac{ a_j }{2} \Big\} \cap Q_1\right|\ge 
\frac{3}{4}
\end{equation}
for large~$j$.
A mirror argument on~$u^+_j$ gives that
\begin{equation}\label{L beta 2}
\left| \Big\{ u^+_j >  
\frac{ a_j }{2} \Big\} \cap Q_1\right|\ge 
\frac{3}{4} 
\end{equation}
for large~$j$.
So, by~\eqref{L beta 1}
and~\eqref{L beta 2},
there must exist~$y'_j$ 
such that~$u^-_j( y'_j) \le -a_j/2$ and~$u^+_j(y'_j)\ge a_j/2$,
hence
$$ u^+_j(y'_j)-u^-_j(y'_j)\ge a_j\ge a/2.$$
This is in contradiction with~\eqref{so long},
and so the proof of Lemma~\ref{ovidiu 6.9}
is completed.
\end{proof}

\subsection{Completion of the proof of Theorem~\ref{MAIN}}

Thanks to Lemma~\ref{ovidiu 6.9},
we have obtained a statement analogous to the one
of Lemma~6.9 of~\cite{CRS}, but with uniform estimates.
Then, the argument from Lemma~6.10 to the end of Section~6
in~\cite{CRS} also yield the proof of Theorem~\ref{MAIN} here.

\section{Proof of Theorem~\ref{-2-}}

The proof is by contradiction.
We suppose that there are~$s_k$-minimal cones~$E_k$
that are not hyperplanes, with~$s_k\rightarrow 1^-$.
By dimensional reduction (see Theorem~10.3 of~\cite{CRS}),
we may focus on the case in which~$E_k$ is singular at
the origin.

{F}rom~\cite{CV}, up to subsequence, we have that~$E_k$
approaches locally uniformly a classical cone of minimal
perimeter. Since~$n\le 7$, we have that such a cone
is a halfspace, say~$\{x_n<0\}$
(see, e.g., Section~1.5.2 of~\cite{Miranda}).
So, for large~$k$, we have that~\eqref{main trap}
holds true for~$E_k$, namely
$$ \partial E_k\cap B_1\subseteq \{ |x\cdot e_n|\le 
\varepsilon_\flat \}.$$
Therefore, by Theorem~\ref{MAIN},
we obtain that~$\partial E_k$
is smooth, i.e.~$E_k$
is a hyperplane, for infinitely many~$k$'s.
This is a contradiction with our assumptions
and it proves 
Theorem~\ref{-2-}.

\section{Proof of Theorem~\ref{-3-}}

Let~$E$ be $s$-minimal. We take the blow up
of~$E$ and we obtain a minimal cone~$E'$ (see Theorem~9.2
of~\cite{CRS}).

By Theorem~\ref{-2-}, we know that~$E'$ is a hyperplane.
Then, $\partial E$ is~$C^{1,\alpha}$, thanks to Theorem~9.4
in~\cite{CRS}. This ends the proof of Theorem~\ref{-3-}.

\section{Proof of Theorems~\ref{GG1} and~\ref{GG2}}

The proofs of Theorems~\ref{GG1} and~\ref{GG2}
follow now verbatim the ones of Theorems~11.7
and~11.8 in~\cite{Giusti} (the only difference
is that the dimensional reduction is
performed via
Theorem~10.3 of~\cite{CRS},
and the regularity needed in low dimension
is assured here by
Theorem~\ref{-2-}).

\end{document}